\documentclass[11pt]{amsart}
\usepackage{amsmath,amscd,amssymb}

\newtheorem{theorem}{Theorem}[section]
\newtheorem{lemma}[theorem]{Lemma}
\newtheorem{proposition}[theorem]{Proposition}
\newtheorem{corollary}[theorem]{Corollary}

\theoremstyle{definition}

\newtheorem{problem}[theorem]{Problem}

\theoremstyle{remark}

\numberwithin{equation}{section}

\def\e{\varepsilon}
\def\CC{{\mathbb C}}
\def\NN{{\mathbb N}}
\def\Int{{\rm Int}\,}
\def\dist{{\rm dist}\,}
\def\conv{{\rm conv}\,}
\def\T{{\mathcal T}}
\def\la{\lambda}
\def\card{{\rm card}\,}

\begin{document}
\title[Matrix representations of bounded operators]{Matrix representations of arbitrary bounded
operators on Hilbert spaces}

\author{Vladimir M\"uller}
\address{Institute of Mathematics,
Czech Academy of Sciences,
ul. \v Zitna 25, Prague,
 Czech Republic}
\email{muller@math.cas.cz}

\author{Yuri Tomilov}
\address{Institute of Mathematics, Polish Academy of Sciences,
\' Sniadeckich str.8, 00-656 Warsaw, Poland}
\address{Faculty of Mathematics and Computer Science, Nicolaus Copernicus University, 
Chopina Str. 12/18, 87-100 Toru\'n, Poland}
\email{ytomilov@impan.pl}



\thanks{The first author was supported
by grant No. 20-22230L of GA CR and RVO:67985840.
The second author was partially supported
by NCN grant UMO-2017/27/B/ST1/00078.}


\begin{abstract}
We show that under natural and quite general 
assumptions, a large part of a matrix for a bounded
linear operator on a Hilbert space can be preassigned.
The result is obtained in a more general setting of operator
tuples leading to interesting consequences, e.g. when the tuple
consists of powers of a single operator. We also prove several variants
of this result of independent interest. The paper substantially 
extends former research on matrix representations in infinite-dimensional spaces dealing
mainly with prescribing the main diagonals.
\end{abstract}

\maketitle
\section{Introduction}
Let $H$ be a separable infinite-dimensional Hilbert space, and let $B(H)$ stand for the space of bounded linear operators
on $H.$
If $T \in B(H),$ then $T$ allows for a variety of matrix representations  $A_T:=\langle T u_j, u_n \rangle_{n, j=1}^\infty$
induced by the set of orthonormal bases $(u_n)_{n=1}^\infty$ in $H.$
In other words, for a fixed basis $(u_n)_{n=1}^\infty$ we study the matrix elements $\langle UT U^{-1}u_j, u_n\rangle$ of the unitary orbit
$\{U TU^{-1}: U \, \text{is unitary}\}$
of $T$ (where then the choice of $(u_n)_{n=1}^{\infty}$ is essentially irrelevant).

While the study of matrix representations goes back to the birth of operator theory,
a number of pertinent facts and insights of their structure were obtained only recently.
In particular, starting from the pioneering works \cite{Kadison02a}, 
\cite{Kadison02b}, 
\cite{Arveson06}, and 
 \cite{Arveson_USA} by Kadison and Arveson on the main diagonals
of projections (and normal operators with finite spectrum), the research on main diagonals of Hilbert space operators attracted a substantial attention.
For sample works in this direction see e.g. \cite{Bownik}, \cite{Bownik1}, \cite{Jasper},
\cite{Kaftal}, \cite{Kennedy},  \cite{Massey},  and the recent survey \cite{LW20}.  
The mainstream of the research on diagonals concentrated around normal operators and their natural subclasses,
consisting 
of unitary and selfadjoint operators, and addressed the problem of 
characterizing sets of possible main diagonals $\mathcal D(T)=\{(\langle T u_n, u_n \rangle)_{n=1}^\infty\}$
for classes of operators of $T \in B(H)$ when $(u_n)_{n=1}^\infty$ varies through the set all orthonormal bases in $H$. In \cite{MT} we've changed this point of view to a more demanding task
of describing the set $\mathcal D(T),$
or at least its substantial subsets, for a fixed $T \in B(H).$
In addition, in \cite{MT}, the problem   
was addressed in a more general  framework of operator tuples $\T=(T_1, \dots, T_k)\in B(H)^k, k \in \mathbb N.$
We relied on the properties of the so-called essential numerical range $W_e(\T)$  
and their relations to the essential spectrum $\sigma_e(\T),$ especially in the case of power tuples $\T=(T, \dots, T^k), \  T \in B(H).$
These properties revealed a new structure in $\mathcal D(T)$ and led in particular  to the so-called non-Blaschke type conditions 
\begin{equation}\label{non-bl}
\sum_{n=1}^\infty {\rm dist}(\lambda_n, \partial W_e(T))=\infty
\end{equation}
on $(\lambda_n)_{n=1}^\infty$ in the interior $\Int\, W_e(T)$ of $W_e(T)$ to be realized by the main diagonal of $T$.
The relevance of $W_e(T)$ in this context 
was first noted by Stout \cite{Stout},  Fan \cite{Fan84} and Herrero \cite{H},
and these works were generalized substantially in \cite{MT}. 
The ideas of \cite{MT} appeared to be fruitful and were further elaborated in \cite{MT_JFA1},
where a systematic approach to matrix representations of bounded operators and operator tuples
was initiated.  
Among other things, we found conditions
for  prescribing three diagonals, for having bands of zeros around the main diagonal, and described several general situations 
where matrix representations with size restrictions
on the set of their matrix elements are possible. In particular, it was shown that if $T \in B(H),$ then $0 \in W_e(\T)$ if and only if for each sequence $(a_n)_{n=1}^\infty\not \in \ell^1(\mathbb N)$ 
there exists an orthonormal basis $(u_n)_{n=1}^{\infty}$ in $H$ such that
$|\langle Tu_j, u_n\rangle| \le \sqrt{|a_{n}a_{j}|}$ for all $j$ and $n,$ see Section 
\ref{section_powers} for more on this. 

However, a more general and natural problem of matching arrays $(a_{nj}),$ $(n,j) \in B,$ with $B \subset \mathbb N\times \mathbb N$ being "large" by a  matrix of a given 
 $T \in B(H)$ 
has been left in \cite{MT_JFA1} widely open. It was not even quite clear whether our methods can handle a band of $(a_{nj})$ consisting from more than three diagonals.
This paper will bridge this gap and show that developing the approach from \cite{MT_JFA1},
one can prescribe the arrays of quite a general nature.

We study the following problem:

\begin{problem}\label{problem}
Let $\T=(T_1,\dots,T_k)\in B(H)^k$, $B \subset \mathbb N \times \mathbb N,$ and $\{a_{nj}: (n,j)\in B\}\subset \mathbb C^k$ be fixed. 
What are natural assumptions on $\T$, $B$ and $(a_{nj})$ to ensure the existence of an orthonormal basis $(u_n)_{n=1}^{\infty}\subset H$ such that
\[
\langle \T u_j,u_n\rangle=
\left( \langle T_1 u_j, u_n\rangle,  \dots, \langle T_k u_j, u_n \rangle \right) =a_{nj}, \qquad (n,j) \in B.
\]
In other words, what is the degree of arbitrariness in a matrix representation of $\T$? 
\end{problem}

Note that a similar problem in the setting of finite-dimensional spaces was studied by a number of authors.

A particular case of Problem \ref{problem} are so-called sparse matrix representations
for $T \in B(H)$ when the arrays $(a_{nj})$ corresponding to admissible sets $B$ of $(n,j)$ consist solely of zeros.
Note that sparse representations appear useful in a number of problems from operator theory.
Recall, in particular, that any $T \in B(H)$ admits a universal block three-diagonal form with an exponential control
on (finite-dimensional) block sizes. Moreover, such blocks can be  further sparsified. To our knowledge, in a full generality,
such block three-diagonal forms appeared first in \cite{Weiss}
and they proved to be crucial
e.g. in the study of commutators \cite{Anders}, \cite{Weiss}, \cite{Loreaux1}, operator norm estimates \cite{Muller_Studia} or properties of diagonal operators \cite{FoWu93}.
Nice accounts of sparse representations can be found in \cite{Pat} and \cite[Sections 4 and 5]{Loreaux1}.
It is instructive to observe that  matrix representations of the opposite kind can be found for any $T\in B(H)$ which is not a scalar multiple of the identity operator: 
by \cite[Theorem 2]{Radjavi}, one can find a basis $(u_n)_{n=1}^\infty \subset H$ such that $\langle T u_j,u_n\rangle \neq 0$ for all $n$ and $j$.
Note also that the issue of sparse representations arises also in the finite-dimensional setting,
where patterns of zeros in a matrix achievable by a unitary transformation are studied.
For sample papers in this direction one may consult  \cite{Doc07}, \cite{HoSc07} and \cite{An10}, 
though we feel that this setting is rather different from the subject of the present paper.

Let $T\in B(H)$ and $(u_n)_{n=1}^\infty$ be an orthonormal basis. It is natural to measure the sparsity of the corresponding matrix $A_T$ 
by the so-called (upper) density given by 
$$
d(A_T):=\limsup_{N \to \infty}N^{-2} {\rm card}\, \{(n,j)\in\NN\times\NN: n,j\le N, \langle Tu_j,u_n\rangle \ne 0\}.
$$
The density is of course bases dependent, and it is of practical interest to have it as small as possible.
Among other things, it was proved in \cite[Corollary 5.7]{Loreaux1} that  for every operator $T\in B(H)$ there is an orthonormal basis $(u_n)_{n=1}^\infty$ in which $A_T$ has density zero.
Our technique allows one to show that in fact much stronger statements hold.
Our sparse representations are quite different from the ones mentioned above, since, in particular,
apart from being sparse in a much stronger sense, the set of their zero elements
can have a comparatively general geometry.

To state our results we need to define several notions describing  size of subsets in $\mathbb N \times \mathbb N.$
They will be basic for all of our considerations to follow.
Denote by $\Delta$ the main diagonal of $\mathbb N\times \mathbb N,$ $\Delta=\{(n,n):n\in\NN\}$.
A set $B\subset \mathbb N\times \mathbb N$ is said to be \emph{subdiagonal} if $B\subset\bigl\{(n,j)\in\NN\times\NN: n>j\bigr\}$. 
We say that a set $B\subset (\mathbb N \times \mathbb N)\setminus\Delta$ is \emph{admissible} if for every $m\in\NN$ there exists $n\in \NN,$ $n >m,$ such that $(j,n)\notin B$ and $(n,j)\notin B$ for all $j=1,\dots,m$.

Clearly a subdiagonal set $B$ is \emph{admissible} if and only if for every $m\in\NN$ there exists $n>m$ such that $(n,j)\notin B$ for all $j=1,\dots,m$.

Note that an admissible set can be quite large. For example, the set 
$$
\{(n,j): n>j\}\setminus\{(2^k,j): k\in\NN, j\le k\}
$$
is an admissible subdiagonal set. Similarly
$$
\{(n,j):n\ne j\}\setminus\{(2^k,n), (n,2^k): k\in\NN, n\le k\}
$$
is admissible.

First, as a warm-up, we study the sparse representations and prove that any tuple of bounded operators has a very sparse matrix representation.

\begin{theorem}\label{sparse}
Let $\T \in B(H)^k$
and $B\subset(\NN\times\NN)\setminus\Delta$ be an admissible set. Then there exists an orthonormal basis $(u_n)_{n=1}^\infty \subset H$ such that
$$
\langle\T u_j,u_n\rangle=0
$$
for all $(n,j)\in B$.
\end{theorem}

Note that in Theorem \ref{sparse} the orthonormal basis $(u_n)_{n=1}^{\infty}$ is common for all operators $T_1,\dots,T_k$.

The condition $B\cap\Delta=\emptyset$ is in general necessary even for single operators. Indeed, if $T\in B(H)$ and the numerical range $W(T)$ of $T$ does not contain zero,
 then the main diagonal of any matrix representation of $T$ consists of non-zero entries.

A better result can be obtained if we assume that $0\in\Int W_e(\T)$. In this situation we can obtain even the zero main diagonal and the matrix representation becomes extremely sparse.
In particular, the folowing statement holds.
\begin{theorem}\label{sparse1}
Let 
$\T \in B(H)^k$
and $0\in\Int W_e(\T)$. Let $f:\NN\to\NN$ be any function satisfying $\lim_{m\to\infty}f(m)=\infty$. Then there exists an orthonormal basis $(u_n)_{n=1}^\infty$ in $H$ such that
$$
\card\bigl\{(n,j)\in\NN\times\NN: n,j\le m,\, \langle \T u_j,u_n\rangle\ne 0\bigr\}\le f(m)
$$
for all $m\in\NN$.
\end{theorem}
Under the assumption  $0 \in \Int W_e(\T)$, the theorem substantially extends 
the results in \cite[Section 5]{Loreaux1}. The assumption is 
natural and optimal as far as one is interested in any increasing $f$.
For a version of Theorem \ref{sparse1} not involving any assumptions on $\T$ see
Theorem \ref{sparse3} below.

Let us now consider Problem \ref{problem} in full generality. 
Under mild assumptions on $T\in B(H)$ we show that large subdiagonal subsets of a matrix $A_T$ 
can be preassigned if the size of the corresponding matrix elements is restricted appropriately.

\begin{theorem}\label{theorem2_intro}
Let $T\in B(H)$ be an operator, which is not of the form
$T=\la I+ K$ for some $\la \in \CC $ and a compact operator
$K\in B(H).$
Then there exists $\delta>0$ {\rm(}depending only on the diameter of $W_e(T)${\rm)} with the following property: if $B\subset\NN\times\NN$ is subdiagonal and admissible, and   
 $\{a_{nj}: (n,j)\in B\}\subset \mathbb C$ satisfy
 \[
\sum_{n:(n,j)\in B}|a_{nj}|\le\delta  \quad  \text{for all}\,\,  j \qquad \text{and} \qquad \sum_{j: (n,j)\in B}|a_{n,j}|\le\delta \quad \text{for all}\, \, n, 
\]
then there is an orthonormal basis $(u_n)_{n=1}^\infty$ in $H$ such that
$$ 
\langle Tu_j,u_n\rangle=a_{nj}
$$
for all $n,j\in\NN$ with $(n,j)\in B$.
\end{theorem}

Theorem \ref{theorem2_intro} can be formulated also for $k$-tuples $\T=(T_1,\dots,T_k)\in B(H)^k$, where none of the operators $T_1,\dots,T_k$ is of the form $\lambda I+K$ with $\lambda\in\CC$ and $K\in B(H)$ compact. However, we preferred to prove its simplified version. 

A technically more involved framework of operator tuples
will be addressed in Theorem \ref{theorem5_intro} below.
Under assumptions  stronger than in Theorem \ref{theorem2_intro},
we prove that large subsets of the whole of $A_T$  can be preassigned 
under size restrictions on the matrix elements similar to those in Theorem \ref{theorem2_intro}.
Note however that one requires additional restrictions on the diagonal elements, which reflects
a special role of the main diagonal in the matrix representations of $T,$ see e.g.  \cite{LW20} and \cite{MT} for more on the topic
of main diagonals.

For $\lambda=(\lambda_1,\dots,\lambda_k)\in\CC^k$ write $\|\lambda\|_\infty=\max\{|\lambda_1|,\dots,|\lambda_k|\}$.

\begin{theorem}\label{theorem5_intro}
Let 
$\T \in B(H)^k$ 
be such that $\Int W_e(\T)\ne\emptyset,$ and let $\e>0$ be fixed. 
Then there exists $\delta>0$ with the following property: if $B\subset(\NN\times\NN)\setminus\Delta$ is admissible, and 
$\{a_{nj}: (n,j)\in B\cup\Delta \}\subset \mathbb C^k$ satisfy :
\begin{itemize}
\item [(i)] $a_{nn}\in\Int W_e(\T)$ and $\dist\{a_{nn},\partial W_e(\T)\}>\e$ for all $n\in\NN$;
\item [(ii)] $\sum_{j:(n,j)\in B}\|a_{nj}\|_\infty\le\delta$
for all $n\in\NN$,
and $\sum_{n: (n,j)\in B}\|a_{jn}\|_\infty\le\delta$ for all $j\in\NN$, 
\end{itemize}
then there is an orthonormal basis $(u_n)_{n=1}^\infty$ in $H$ such that
$$ 
\langle \T u_j,u_n\rangle=a_{nj}
$$
for all $n,j\in\NN$ with $(n,j)\in B\cup\Delta$.
\end{theorem}
Theorem \ref{theorem5_intro} is the main result of this paper, and we are not aware of any
similar results in the literature. It is plausible that the assumption (i) in Theorem \ref{theorem5_intro}  can be replaced by a multivariate counterpart of the non-Blaschke type condition \eqref{non-bl} (with $\lambda_n$ replaced by $a_{nn}$).
However, the methods of the paper  lead to substantial technical complications on this way, and should probably be refined.

Clearly for any $m\in\NN$, the set $\{(n,j): 1\le|n-j|\le m\}$ is an admissible set. So in particular we can prescribe any finite number of diagonals in the matrix representation of $\T$
subject to mild restrictions on absolute values of their elements.
We formulate this conclusion as a separate statement generalising essentially \cite[Theorem 2.4]{MT_JFA1}.
\begin{corollary}\label{cor1}
Let 
$\T \in B(H)^k$
be such that $\Int W_e(\T)\ne\emptyset,$ and let $\e>0$ and $m \in \mathbb N$ be fixed. 
Then there exists $\delta=\delta(k, m, \e)>0$ such that 
if $\{a_{nj}: |n-j|\le m \}\subset \mathbb C^k$ satisfy :
\begin{itemize}
\item [(i)] $a_{nn}\in\Int W_e(\T)$ and $\dist\{a_{nn},\partial W_e(\T)\}>\e$ for all $n\in\NN$;
\item [(ii)] $\sup \{\|a_{nj}\|_\infty: 1\le |n-j| \le m \}\le\delta;$
\end{itemize}
then there is an orthonormal basis $(u_n)_{n=1}^\infty$ in $H$ such that
$$ 
\langle \T u_j,u_n\rangle=a_{nj}, \qquad  |n-j|\le m.
$$
\end{corollary}

The framework of operator tuples makes it possible to formulate similar  results 
for tuples of powers  $\T=(T, T^2,\dots, T^k)$ under the spectral assumptions on $T \in B(H)$ 
rather than the assumptions on $W_e(\T)$ as above,
making the obtained results more explicit.

If $T\in B(H)$ is such that $0\in\Int\hat\sigma(T),$ where $\hat \sigma(T)$ stands for the polynomial hull of $\sigma(T),$ then $0\in\Int W_e(T,T^2,\dots,T^k)$ for all $k\in \NN$ (see Section \ref{prelim}). So we can prescribe quite large subset of entries simultaneously for any finite number of powers $T^j$.


Similar results are also proved in the case of an invertible operator $T\in B(H)$ and for tuples consisting of both positive and negative powers of $T$.
In this case we impose a stronger assumption $r\mathbb T\cup s\mathbb T\subset\sigma_e(T)$, where $0<r<s$ and $\mathbb T$ stands for the unit circle.

We finish this section with a discussion of the notion of admissibility of arrays in 
$\mathbb N\times \mathbb  N,$ crucial for the rest of this paper. 
While the admissibility allows for rather arbitrary geometries of 
zero sets for sparse
matrix representations,  it is far from being necessary, in general.
Indeed, recall that $T \in B(H)$ is cyclic if and only if $T$ admits an upper ``triangular+$1$''
matrix representation $(a_{n,j})_{n,j=1}^\infty$  such that the diagonal $(a_{(j+1), j})_{j=1}^\infty$ right below the main diagonal has no nonzero entries,
see e.g. \cite[p. 285-286]{Halmos}. (A related representation for arbitrary $T \in B(H)$ can be found in \cite[Theorem 5]{Douglas}.)
Clearly, for such a representation  the 
 set of $(n,j)$ with $a_{nj}=0$ is not admissible.
Moreover, the admissibility fails for universal block three-diagonal representations for arbitrary $T \in B(H)$
discussed above.

On the other hand, the set of zeros in the matrix representation cannot in general be too ``small'', 
outside a band around the main diagonal.
Observe that if $T$ is a bounded normal operator on $H$ with absolutely continuous spectral measure, then
$T$ cannot be represented by a matrix with finite number of non-zero diagonals, see e.g. 
\cite[Theorem 3.9, (a)]{FoWu96} or  \cite[Corollary 4]{Shulman}.
However, the infinite number of non-zero diagonals does not, in general, exclude the admissibility 
of arrays corresponding to zero elements.
If $M$ is the multiplication operator on $L^2$ over the unit disc $\mathbb D$, given by
 $(Mf)(z)=zf(z), f \in L^2(\mathbb D),$ then $W_e(M)=\sigma_e(M)=\overline{\mathbb D},$ so
$M$ fits into each of the two examples mentioned above, and satisfies the extra assumptions
of our results on sparse representations.

The setting of  Theorems \ref{theorem2_intro} and \ref{theorem5_intro} providing prescribed arrays 
in matrix representations of bounded operators and their tuples is more demanding and quite new. Constraints on the given arrays there is still enigmatic and requires further research.
\section{Preliminaries and notations}\label{prelim}
\subsection{The relevance of numerical ranges}

First, we recall some standard notation used in the context of operator tuples.
For a $k$-tuple $\mathcal T=(T_1,\dots,T_k)\in B(H)^k, k \in \mathbb N,$ and $x,y\in H$ we write shortly
$$\langle \mathcal Tx,y\rangle= (\langle T_1x,y\rangle,\dots,\langle T_kx,y\rangle)\in\CC^k  \quad \text{and} \quad {\mathcal T}x=(T_1x,\dots,T_kx)\in H^k.$$
If $\la=(\la_1,\dots,\la_k)\in\CC^k$
we denote $\mathcal T-\la=(T_1-\la_1,\dots,T-\la_k)$ and
\begin{equation}\label{lam}
\|\lambda\|_\infty=\max\{|\lambda_1|,\dots,|\lambda_k|\}.
\end{equation}

In our studies of matrix representations for a bounded operator $T$ on
$H$ and more generally for operator tuples $\T \in B(H)^k$, we will rely on the well-studied  notions of the (joint) numerical range $W(\T),$ 
given by
$$
W(\mathcal T)=\bigl\{(\langle T_1 x, x\rangle , ..., \langle T_k x, x \rangle) : x \in H, \|x\|=1\bigr\},
$$
and of the essential numerical range $W_e(\T)$ of $\T.$
Being an approximate version of $W(\T),$ the latter notion allows for several equivalent definitions.
To fix one of them, for
$\mathcal T=(T_1,\dots,T_k)\in B(H)^k$ we define  the (joint) essential numerical range   $W_{e}(\mathcal T)$ of $\T$ as the set of all $k$-tuples
$\la=(\la_1,\dots,\la_k)\in\mathbb C^k$ such that there exists an orthonormal sequence $(u_n)_{n=1}^{\infty}\subset H$ with
$$
\lim_{n\to\infty}\langle  T_t u_n,  u_n\rangle=\la_t, \qquad t=1,\dots,k.
$$
Recall that $W_{e}(\mathcal T)$ is a nonempty, compact and, in contrast to $W(\mathcal T),$ \emph{convex} subset of $\overline{W(\mathcal T)}$, see  e.g. \cite{Li-Poon}. 
At the same time, $W(\mathcal T)$ is convex if $k=1,$ and it may be non-convex if $k >1.$
Moreover, even if $k=1,$ then it may be the case that $W(T)$ is neither closed nor open.

The next properties of $W_e(\T)$ and $W(\T)$ are crucial for the sequel
and will be used frequently.

\begin{lemma}\label{codim}
Let $\T \in B(H)^k.$
\begin{itemize}
\item [(a)]  One has $\lambda \in W_{e}(\T)$ if and only if
for every $\e>0$ and every  subspace $M\subset H$ of finite codimension
there is a unit vector $x\in M$ such that
$$\|\langle\T x,x\rangle-\lambda\|_{\infty}<\e.$$
\item [(b)]
If $\lambda \in\Int \, W_e(\mathcal T),$
then for every subspace $M\subset H$ of a finite codimension there is $x\in M$ such that $\|x\|=1$ and
$$
 \langle \T x, x\rangle=\lambda.
$$
\end{itemize}
\end{lemma}

The proof of the first property is easy and can be found e.g. in \cite[Proposition 5.5]{MT_LMS}.
Since $W(\T)$ is not in general convex, the second property is more involved,
see \cite[Corollary 4.5]{MT_JFA} for its proof and other related statements. 
The properties (a) and (b) are very useful in inductive constructions of 
sequences in $H.$ In particular,
by absorbing all of the elements constructed after a finite number of induction steps
into a finite-dimensional subspace $F$, 
one may still use $W_e(\T)$ when dealing with vectors from $F^\perp.$ 
The properties will play a similar role  in this paper.

To note another usage of numerical ranges, recall that the joint spectrum of a commuting tuple $\T=(T_1, \dots, T_k)$ can be defined as 
 the (Harte) spectrum of the $n$-tuple $(T_1,\dots,T_k)$ in the algebra $B(H).$
Similarly, the joint essential spectrum $\sigma_e(\mathcal T)$ is defined as the (Harte) spectrum of the $k$-tuple $(T_1+{\mathcal K}(H),\dots,T_k+{\mathcal K}(H))$ in the Calkin algebra $B(H)/{\mathcal K}(H)$, where ${\mathcal K}(H)$ denotes the ideal of all compact operators on $H$. One of the main features of $W(\T)$ and $W_e(\T)$ is that in view of the inclusions
$\sigma(\T)\subset \overline{W(\T)}$ and $\sigma_e(\T)\subset W_e(T)$ 
these numerical ranges help to localize spectrum. Sometimes, when the spectral information is more accessible, one may argue the other way round
and to identify big subsets of $W(T)$ and $W_e(T)$ in spectral terms. In particular, this becomes apparent for tuples $\T$ of special form
$\T=(T, \dots, T^k), T \in B(H).$ Note that $\sigma (\mathcal T)=\{(\lambda, \dots, \lambda^k): \lambda \in \sigma (T)\}$ and
$\sigma_{e} (\mathcal T)=\{(\lambda, \dots, \lambda^k): \lambda \in \sigma_e (T)\}$ (cf. Section \ref{section_powers})
so that the spectral properties of the tuple $(T, \dots, T^k)$ are determined by the spectral properties of $T$.
The relevance of spectrum for the study of numerical ranges can be illustrated by the next "numerical ranges" mapping theorem \cite[Theorem 4.6]{MT_JFA},
 important for the sequel (see Section \ref{section_powers}).
To formulate it, recall that if $K\subset\CC$ is compact, then the polynomial hull 
$\widehat K:=\{\lambda \in \mathbb C: |p(\lambda)|\le \sup_{z\in K}|p(z)| \,\, \text{for all polynomials}\, p\}$ of $K$ can be described 
as the union of $K$ with all bounded components of the complement $\mathbb C\setminus K$.
 If ${\rm conv}\, K$ stands for the convex hull
of $K,$ then clearly $\widehat K \subset {\rm conv}\, K$. 

\begin{theorem}\label{spectrum_polyn}
Let $T\in B(H).$ If  $\la 
\in \Int \widehat \sigma (T),$   
then
$$
(\la,\la^2,\dots,\la^k)\in \conv\,\{(z,z^2,\dots,z^k):z\in\sigma_e(T)\}\subset \Int W_{e}(T, T^2, \dots, T^k)
$$
for all $k\in\NN.$
\end{theorem}

More information on joint essential numerical range for operator tuples and its relation to spectral theory can be found in \cite{MT_LMS} and \cite{MT},
see also \cite{Li-Poon}. The classical case $k=1$ is considered in detail in \cite{Bonsall} and \cite{Fillmore}.
\subsection{Some notations}

Let $T\in B(H)$ and $u,v\in H$. We write for short $u\perp^{(T)}v$ if $u\perp v, Tv, T^*v$. More generally, if $\T=(T_1,\dots,T_k)\in B(H)^k$ then we write $u\perp^{(\T)}v$ if
$$
u\perp v, T_1v,\dots,T_kv,T_1^*v,\dots,T^*_kv.
$$
Clearly $u\perp^{(\T)}v$ if and only if $v\perp^{(\T)} u$. Note that if $\T=(T_1,\dots,T_k)\in B(H)^k$ and $v_1,\dots,v_m\in H$ then the set
$$
\{u\in H: u\perp^{(\T)} v_1,\dots,v_m\}
$$
is a subspace of $H$ of finite codimension.

For a subspace $L\subset H$ denote by $P_L$ the orthogonal projection onto $L$.

To simplify our presentation, we do not distinguish between $0$ and $(0,\dots, 0)$
and write just $0$ whenever the vector notation is formally needed.

As above, for a compact subset $K \subset \mathbb C^n$ we denote by $\Int K$ the interior of $K,$  by $\partial K$ the topological boundary of $K,$
by $\conv K$ the convex hull of $K$ and by $\widehat K$ the polynomial hull of $K.$

\section{Sparse representations}
First we prove Theorem \ref{sparse} that any tuple of operators has a very sparse matrix representation.
\bigskip

\begin{proof}[Proof of Theorem \ref{sparse}.]
Let $\T=(T_1,\dots,T_k)\in B(H)^k$ be any $k$-tuple of operators and $B\subset(\NN\times\NN)\setminus\Delta$ an admissible set.
We show that there exists an orthonormal basis $(u_n)_{n=1}^\infty$ in $H$ such that
$$
\langle \T u_j,u_n\rangle=0
$$
for all $(n,j)\in B$. 

Let $n_0=0$ and construct inductively an increasing sequence of integers $(n_s)_{s=0}^\infty$ such that
$$
(j,n_s)\notin B\qquad\hbox{and}\qquad(n_s,j)\notin B,\qquad j=1,\dots,n_{s-1}.
$$
Let $(y_r)_{r=1}^\infty$ be a sequence of vectors in $H$ such that $\bigvee_{r=1}^\infty y_r=H$.

We construct the vectors $u_n$ inductively.

Let $s\ge 1$ and suppose that orthonormal vectors $u_1,\dots,u_{n_{s-1}}\in H$ satisfy
\begin{itemize}
\item[(i)] $\langle T_tu_n,u_j\rangle=0$ for all $t=1,\dots,k$ and all $(n,j)\in B$ with $n,j\le n_{s-1}$;
\item[(ii)] $y_r\in\bigvee_{n=1}^{n_{r}}u_n$ for all $r\le s-1$.
\end{itemize}

For all $n,\ \ n_{s-1}<n<n_s,$ find inductively unit vectors $u_n$ such that
$$
u_n\perp^{(\T)}u_1,\dots,u_{n-1},y_s.
$$
Then $\langle\T u_n,u_m\rangle=\langle\T u_m,u_n\rangle=0$ for all $m\le n-1$, $n_{s-1}<n<n_s$.

In order to construct $u_{n_s}$ we distinguish two cases. If $y_s\in\bigvee_{j=1}^{n_{s-1}}u_j$ then choose $u_{n_s}$ any unit vector satisfying
$$
u_{n_s}\perp^{(\T)}u_1,\dots,u_{n_s-1}.
$$
If $y_s\notin\bigvee_{j=1}^{n_{s-1}}u_j$ then set
$$
u_{n_s}=\frac{(I-P_{M_{n_{s-1}}})y_s}{\|(I-P_{M_{n_{s-1}}})y_s\|},
$$
where $P_{M_{n_{s-1}}}$ is the orthogonal projection onto the subspace $M_{n_{s-1}}:=\bigvee_{j=1}^{n_{s-1}}u_j$. Clearly 
$\|u_{n_s}\|=1$ and $u_{n_s}\perp u_1,\dots,u_{n_{s-1}}$. Moreover, 
$$
u_{n_s}\in\bigvee\{y_s,u_1,\dots,u_{n_{s-1}}\}\perp^{(\T)} u_m
$$
for all $m, n_{s-1}<m<n_s$ by the construction. So the set $\{u_1,\dots,u_{n_s}\}$ is orthonormal.

If $m<n_s$ and either $(m,n_s)\in B$ or $(n_s,m)\in B$ then $m>n_{s-1}$ and
$$
T_t u_{n_s}\in\bigvee\{T_t y_s,T_t u_1,\dots,T_t u_{n_{s-1}}\}\subset u_m^\perp
$$
for all $t=1,\dots,k$.
So
$\langle \T u_{n_s},u_m\rangle=0$.
Similarly, $\langle\T u_m,u_{n_s}\rangle=0$.

Moreover, we have $y_s\in\bigvee_{j=1}^{n_s}u_j$.

If we construct the vectors $u_n, n\in\NN,$ in this way then they will form an orthonormal system satisfying
$$
\langle \T u_j,u_n\rangle=0
$$
for all $(n,j)\in B$. Moreover, $y_r\in\bigvee_{n=1}^\infty u_n$ for all $r$, and so $(u_n)_{n=1}^\infty$ form an orthonormal basis.

\end{proof}

As mentioned in the introduction, 
the assumption that $B\subset(\NN\times\NN)\setminus\Delta$ cannot be in general omitted. In general, 
all entries on the main diagonal may be non-zero for any choice of an orthonormal basis if $0\not \in W(T)$
(e.g. if ${\rm Re}\, T\ge c I, c >0$).

If we assume that $0\in\Int W_e(\T)$ then it is possible to obtain also the zero main diagonal.
The next result is a consequence of a more general Theorem \ref{theorem5_intro}. However, we give a direct proof because it is much simpler
and, at the same time, contains all of the main ideas behind the proof of Theorem \ref{theorem5_intro}.

\begin{theorem} \label{sparse2}
Let 
$\T \in B(H)^k$ 
satisfy $0\in\Int W_e(\T)$  
and let $B\subset (\NN\times\NN)\setminus\Delta$ be an admissible set. Then there exists an orthonormal basis $(u_n)_{n=1}^\infty\subset H$ such that
\begin{equation}\label{sp}
\langle \T u_j,u_n\rangle=0
\end{equation}
for $(n,j)\in B\cup\Delta$. 
\end{theorem}

\begin{proof}
Let $\T=(T_1,\dots,T_k)\in B(H)^k.$
Without loss of generality we may assume that $\|T_t\|\le 1$ for all $t=1,\dots,k$.

Let $n_0=0$ and construct inductively an increasing sequence of integers $(n_s)_{s=0}^\infty$ such that
$$
(j,n_s)\notin B\qquad\hbox{and}\qquad(n_s,j)\notin B,\qquad j=1,\dots,n_{s-1}.
$$

Fix a number $\eta\in(0,1)$ such that
$$
\frac{\eta}{1-\eta}<\dist\{0,\partial W_e(\T)\}.
$$
Fix a sequence of unit vectors $(y_r)_{r=0}^\infty$ in $H$ such that $\bigvee_{r=0}^\infty y_r=H$.

Each $s\in\NN$ can be written as $s=2^{r(s)}(2l(s)-1)$ where $r(s)\ge 0$ and $l(s)\ge 1$ are uniquely determined integers.

We construct the vectors $u_n$ inductively.

Let $s\ge 1$ and suppose that orthonormal vectors $u_1,\dots,u_{n_{s-1}}\in H$ satisfying
\begin{itemize}
\item[(i)] $\langle \T u_j,u_j\rangle=0$ for all $j=1,\dots,n_{s-1}$;
\item[(ii)] $\langle T_tu_n,u_j\rangle=0$ for all $t=1,\dots,k$ and all $(n,j)\in B$ with $n\ne j$ and $n,j\le n_{s-1}$;
\item[(iii)] $\dist^2\Bigl\{y_r,\bigvee_{j=1}^{n_{2^r(2l-1)}}u_j\Bigr\}\le (1-\eta)^l$ for all $r,l$ with $2^r(2l-1)\le s-1$.
\end{itemize}

For $n,\ \ n_{s-1}<n<n_s,$ using Lemma \ref{codim}, (b), find inductively unit vectors $u_n$ such that
$$
u_n\perp^{(\T)}u_1,\dots,u_{n-1},y_{r(s)}
$$
and
$$
\langle\T u_n,u_n\rangle =0.
$$
Then $\langle\T u_n,u_m\rangle=\langle\T u_m,u_n\rangle=0$ for all $m\le n$, $n_{s-1}<n<n_s$.

In order to construct $u_{n_s}$ we distinguish two cases. If $y_{r(s)}\in\bigvee_{j=1}^{n_{s-1}}u_j$ then let $u_{n_s}$ be any unit vector
satisfying
$$
u_{n_s}\perp^{(\T)}u_1,\dots,u_{n_s-1}
$$
and
$$
\langle\T u_{n_s},u_{n_s}\rangle =0.
$$
Then clearly (i)--(iii) are satisfied.

If $y_{r(s)}\notin\bigvee_{j=1}^{n_{s-1}}u_j$ then set
$$
b_{n_s}=\frac{(I-P_{M_{n_{s-1}}})y_{r(s)}}{\|(I-P_{M_{n_{s-1}}})y_{r(s)}\|},
$$
where $P_{M_{n_{s-1}}}$ is the orthogonal projection onto the subspace $M_{n_{s-1}}:=\bigvee_{j=1}^{n_{s-1}}u_j$.
We have 
$$
\Bigl|\frac{\eta}{1-\eta}\langle \T b_{n_s},b_{n_s}\rangle\Bigr|<\dist\{0,\partial W_e(\T)\},
$$
so by Lemma \ref{codim}, (b) there exists a unit vector $v_{n_s}\in H$ such that
$$
v_{n_s}\perp^{(\T)}u_1,\dots,u_{n-1},b_{n_s}
$$
and
$$
\langle\T v_{n_s},v_{n_s}\rangle= -\frac{\eta}{1-\eta}\langle \T b_{n_s},b_{n_s}\rangle.
$$
Define
$$
u_{n_s}=\sqrt{1-\eta}\,v_{n_s}+\sqrt{\eta} \,b_{n_s}.
$$
Clearly $\|u_{n_s}\|=1$ since $v_{n_s}\perp b_{n_s}$.

Clearly $u_{n_s}\perp u_1,\dots,u_{n_{s-1}}$. For $j, n_{s-1}<j<n_s,$ we have
$$
\langle u_{n_s},u_j\rangle= \langle \sqrt \eta b_{n_s},u_j\rangle=0
$$
since $b_{n_s}\in\bigvee\{y_{r(s)},u_1,\dots,u_{n_{s-1}}\}\subset u_j^\perp.$
So the vectors $u_1,\dots,u_{n_s}$ are orthonormal.

We have
$$
\langle \T u_{n_s},u_{n_s}\rangle=(1-\eta)\langle \T v_{n_s},v_{n_s}\rangle+\eta \langle \T b_{n_s},b_{n_s}\rangle=0.
$$
If $j< n_s$ and $(n_s,j)\in B$ then $j>n_{s-1}$ and $\langle\T u_j,u_{n_s}\rangle=
\langle \T u_j, \sqrt \eta b_{n_s}\rangle=0$
since
$$
b_{n_s}\in\bigvee\{y_{r(s)},u_1,\dots,u_{n_{s-1}}\}\subset\perp^{(\T)} u_j.
$$
Similarly, $\langle\T u_{n_s},u_j\rangle=0$ if $j< n_s$ and $(j,n_s)\in B$.

Finally,
\begin{align*}
\dist^2\{y_{r(s)}, M_{n_s}\}=&
\dist^2\{y_{r(s)},M_{n_s-1}\}-|\langle y_{r(s)},u_{n_s}\rangle|^2\\
\le&
\dist^2\{y_{r(s)},M_{n_{s-1}}\}-|\langle y_{r(s)},\sqrt \eta b_{n_s}\rangle|^2\\
=&
\dist^2\{y_{r(s)},M_{n_{s-1}}\}(1-\eta)
\le
(1-\eta)^{l(s)}
\end{align*}
by the induction assumption.

Suppose that the vectors $u_n, n\in\NN,$ have been constructed in the way described above. Then the vectors $(u_n)_{n\in\NN}$ form an orthonormal system satisfying
$$
\langle \T u_j,u_n\rangle=0
$$
for all $(n,j)\in B\cup\Delta$. Moreover, for each $r\ge 0$ we have
$$
\dist^2\Bigl\{y_r,\bigvee_{j=1}^\infty u_j\Bigr\}=
\lim_{l\to\infty}\dist^2\bigl\{y_r, M_{2^r(2l-1)}\bigr\}
\le
\lim_{l\to\infty}(1-\eta)^l=0.
$$
So $y_r\in\bigvee_{j=1}^\infty u_j$. Since $\bigvee_{r=0}^\infty y_r=H$, the vectors $(u_n)_{n=1}^\infty$ form an orthonormal basis.
\end{proof}

Theorem \ref{sparse2} implies that operators $T\in B(H)$ with $0\in\Int W_e(T)$ have extremely sparse representations
as stated in Theorem \ref{sparse1} given in the introduction.

\begin{proof}[Proof of Theorem \ref{sparse1}] 
Let $\T=(T_1,\dots,T_k)\in B(H)^k.$
Without loss of generality we may assume that the function $f$ is nondecreasing (if not, then replace $f(m)$ by $\inf\{f(j):j\ge m\}$).

Find an increasing sequence $(n_k)_{k=1}^\infty$ such that $f(n_k)\ge (k+1)^2$ for each $k\in\NN$.

Let 
$$
B=\NN\times\NN\setminus \bigl\{(n_k,j), (j,n_k): k\in\NN, j\le k\bigr\}.
$$
For each $m$, $n_k\le m< n_{k+1}$, we have
$$
\card\bigl\{(n,j): n,j\le m, (n,j)\notin B\bigr\}\le \sum_{r=1}^k (2r)\le (k+1)^2\le f(n_k)\le f(m).
$$
Clearly $B\setminus\Delta$ is an admissible set. By Theorem \ref{sparse2}, there exists an orthonormal basis $(u_n)_{n=1}^\infty$ such that $\langle \T u_j,u_n\rangle=0$ for all $(n,j)\in B$.
\end{proof}

Clearly the condition $f(m)\to\infty$ as $m \to \infty$ is in general necessary. It is easy to see that there exists a matrix representation of $T\in B(H)$ with $\card\{(n,j): \langle Tu_n,u_j\rangle\ne 0\}<\infty$ if an only if $T$ is a finite rank operator.

If we do not assume that $0\in\Int W_e(\T)$ then in general all entries on the main diagonal may be non-zero for all orthonormal bases.  So we can state the next version of Theorem \ref{sparse1} (having the same proof).

\begin{theorem}\label{sparse3}
Let 
$\T \in B(H)^k.$ 
Let $f:\NN\to\NN$ be any function satisfying $\lim_{m\to\infty}f(m)=\infty$. Then there exists an orthonormal basis $(u_n)_{n=1}^\infty$ in $H$ such that
$$
\card\bigl\{(n,j)\in\NN\times\NN: n,j\le m,\, \langle \T u_j,u_n\rangle\ne 0\bigr\}\le m+f(m)
$$
for all $m\in\NN$.
\end{theorem}

Theorem \ref{sparse2} applies directly to $k$-tuples of the form $(T,T^2,\dots,T^k)$ where $T\in B(H)$ satisfies $0\in\Int\hat\sigma(T)$. We discuss this in the last section.

\section{Prescribing subdiagonal entries}
Recall that
$T\in B(H)$ is compact if and only if $W_{e}(T)=\{0\}$. So $T$ is of the
form $T=\la I+K$ for some $\la \in \CC $ and a compact operator
$K\in B(H)$ if and only if $W_{e}(T)$ is a singleton, i.e.,
\begin{equation*}
{\rm diam} W_{e}(T)=\max \{|\la -\mu |: \la,\mu \in W_{e}(T)\}=0.
\end{equation*}

The next lemma will be crucial in the inductive construction leading to Theorem \ref{theorem2_intro}.
Similar statements can be found in \cite[Section 2,3]{Brown} and \cite[Section 5]{MT_JFA1}.

\begin{lemma}\label{lemma1}
Let $T\in B(H)$ be an operator, which is not of the form
$T=\la I+ K$ for some $\la \in \CC $ and a compact operator
$K\in B(H),$ and let
$0<C<{\rm diam}\, W_e(T)$. Let $M\subset H$ be a subspace of finite codimension.
Then there exist vectors $v,z\in M$ such that 
\[ \|v\|\cdot\|z\|\le \frac{2\sqrt{2}}{C}, \quad v\perp z,\quad \text{and}\quad  \langle Tv,z\rangle=1.
\]
\end{lemma}

\begin{proof}
Let $\lambda,\mu\in W_e(T)$ be such that $|\lambda-\mu|={\rm diam}\, W_e(T)$. Choose a positive number $\e$ such that $\e<({\rm diam}\, W_e(T) -C)/4.$
Using Lemma \ref{codim}, (a),
let $x\in M\cap T^{-1}M$ be a unit vector such that $|\langle Tx,x\rangle-\lambda|<\e$.

Similarly, let $y\in M\cap T^{-1}M$, $y\perp^{(T)} x$ be a unit vector such that $|\langle Ty,y\rangle-\mu|<\e$.

Let $v=\frac{x+y}{\sqrt{2}}$. Then 
$$v\in M\cap T^{-1}M, \qquad \|v\|=1, \qquad \langle Tv,v\rangle=\frac{1}{2}\bigl(\langle Tx,x\rangle+\langle Ty,y\rangle\bigr),
$$ 
and so
$$
\Bigl|\langle Tv,v\rangle-\frac{\lambda+\mu}{2}\Bigr|<\e.
$$
Let $$w=Tv-\langle Tv,v\rangle v.$$
 Then $w\in M$, $w\perp v$ and
\begin{align*}
\|w\|\ge |\langle w,x\rangle|=&
\bigl|\langle Tv,x\rangle-\langle Tv,v\rangle\langle v,x\rangle\bigr|=
\frac{1}{\sqrt{2}}\Bigl|\langle Tx,x\rangle-\langle Tv,v\rangle\Bigr|\\
\ge&
\frac{1}{\sqrt{2}}\Bigl(\Bigl|\lambda-\frac{\lambda+\mu}{2}\Bigr|-2\e\Bigr)=
\frac{1}{\sqrt{2}}\frac{|\lambda-\mu|}{2}-\frac{2\e}{\sqrt{2}}>\frac{C}{2\sqrt{2}}.
\end{align*}
Set \[ z=\frac{w}{\|w\|^2}.\]
 Then 
\[z\in M, \quad z\perp v, \quad \|z\|=\frac{1}{\|w\|}\le\frac{2\sqrt{2}}{C} \quad \text{and}\quad
\langle Tv,z\rangle=\langle w,z\rangle=1,
\]
as required.
\end{proof}

The proof of Theorem \ref{theorem2_intro} is less technically demanding than the proof of its full matrix analogue, Theorem \ref{theorem5_intro},
 and it thus provides a good intuition
needed for understanding a more involved argument for Theorem \ref{theorem5_intro} in the next section.

\begin{proof}[Proof of Theorem \ref{theorem2_intro}.]
 Choose positive numbers $C$ and $\delta$ such that $C<{\rm diam}\, W_e(T),$ and 
$$
\delta <  \frac{C}{4\sqrt{2}}.
$$
For a given subdiagonal and admissible set $B\subset\NN\times\NN$ and the corresponding array $\{a_{nj}: (n,j)\in B\}$
subject to the size restrictions
 \[
\sum_{n:(n,j)\in B}|a_{nj}|\le\delta  \quad  \text{for all}\,\,  j \qquad \text{and} \qquad \sum_{j: (n,j)\in B}|a_{n,j}|\le\delta \quad \text{for all}\, \, n, 
\]
we construct an orthonormal basis $(u_n)_{n=1}^\infty \subset H$ such that
\begin{equation}\label{matrix1}
\langle Tu_j,u_n\rangle=a_{nj} \qquad \text{for all}\,\, n,j\in\NN,\, (n,j)\in B.
\end{equation}

To clarify technical details, the proof will be divided into several steps. 
The construction of $(u_n)_{n=1}^{\infty}$ will be based on several inductive arguments.
We start with an appropriate choice of parameters needed for our subsequent considerations.

Fix the next initial settings:
\begin{itemize}
\item [(i)] 
Fix a positive number $\eta$ such that 
$\eta <1- \frac{4\delta\sqrt{2}}{C}$.

\item [(ii)] For $s\in\NN$ write  $s=2^{r(s)}(2l(s)-1)$, where the integers $r(s)\ge 0$ and $l(s)\ge 1$ are defined uniquely.

\item [(iii)] Fix (very small) positive numbers $\rho_s,  s\in\NN$. The precise size of these numbers is not important, we require only
that $\sum_{s=1}^\infty\rho_s<1$ and 
\begin{equation}\label{dist_bound}
\bigl((1-\eta/2)^{(l(s)-2)/2}+\rho_s\bigr)^2(1-\eta)\le (1-\eta/2)^{l(s)-1}
\end{equation}
for every $s\in \NN$ represented as in (ii).

\item [(iv)] Set formally $n_0=0$ and define inductively an increasing sequence $(n_s)_{s=1}^\infty$ such that $n_1=1$ and $(n_s,m)\notin B$ for all $m=1,\dots,n_{s-1}$. 

\item [(v)] For $(n,i)\in B$ let 
$$\beta_{ni}:=|a_{ni}|^{1/2}\arg(a_{ni}) \qquad \text{and} \qquad  \gamma_{ni}:=|a_{ni}|^{1/2}.$$
\end{itemize}
The vectors $u_n, n\in\NN,$ forming an orthonormal basis in $H,$ will be constructed in the form
$$
u_n=\alpha_nw_n+\sum_{j:(j,n)\in B}\beta_{jn}v_{jn}+\sum_{i:(n,i)\in B}\gamma_{ni}z_{ni}+b_n,
$$
where $w_n, v_{jn}, z_{ni}$ and $b_n$ are suitable  elements of $H$. Each of the pairs $v_{jn},z_{jn}$ will ensure that $\langle Tu_j,u_n\rangle=a_{nj}$ for $(n,j) \in B.$ The vectors $w_n$ are used only to have $\|u_n\|=1$.
The vectors $b_n$ will help to arrange $\bigvee_{n=1}^{\infty} u_n=H.$ To this end, we fix in advance an orthonormal basis $(y_r)_{r=0}^\infty$ in $H$, 
and for every $r$ we construct a sequence $\{y_{r,l}:l \ge 0\}$ such that $\lim_{l\to\infty}y_{r,l}=y'_r$ with  $y'_r\in\bigvee_{n=1}^{\infty} u_n$ and $y'_r$ being close enough to $y_r$. This will imply that  $\bigvee_{r=0}^\infty y'_r=H$ and then $\bigvee_{n=1}^{\infty}u_n=H$.
Set formally $y_{r,0}=y_r$ for all $r$.

First, by an inductive argument, we construct vectors $b_n,w_n, n\in\NN,$ and $v_{ni},z_{ni},  (n,i)\in B$, and  numbers $\alpha_n\ge 0$ 
in the following way:

Let $s\in\NN$ and suppose that the vectors 
$$
b_n, \,\, w_n, \,\, v_{ni}, \,\, z_{ni}, \quad n\le n_{s-1}, (n,i)\in B
$$
$$
y_{r,l}, \quad 2^r(2l-1)\le s-1,
$$ 
and  numbers $\alpha_n,\ \  n\le n_{s-1},$ 
have already been constructed in such a way
that if
$$
u_{n,s-1}:=\alpha_nw_n+\sum_{j:(j,n)\in B\atop j\le n_{s-1}}\beta_{jn}v_{jn}+\sum_{i:(n,i)\in B}\gamma_{ni}z_{ni}+b_n, \qquad n=1,\dots,n_{s-1},
$$
then the vectors $u_{1,s-1},\dots,u_{n_{s-1},s-1}$ are mutually orthogonal, 
$$
\|u_{n,s-1}\|^2=1-\sum_{j>n_{s-1}, (j,n)\in B}|a_{jn}|\frac{2\sqrt{2}}{C}, \qquad n=1,\dots,n_{s-1},
$$
and
$$ \|y_{r,l}-y_{r,l-1}\|\le\rho_{2^r(2l-1)} \quad \text{for all}\ \ r,l \ \ \text{with}\ \ 2^r(2l-1)\le s-1.$$
\medskip

(A) Define first vectors $b_n,  n_{s-1}<n\le n_s.$ 

If $n_{s-1}<n<n_s,$ then set $b_n=0$. 

Write  $s=2^{r(s)}(2l(s)-1)$ and  
define 
$$
L_{s-1}=\bigvee_{n=1}^{n_{s-1}}u_{n,s-1}.
$$
 If $y_{r(s),l(s)-1}\notin L_{s-1}$ then set $y_{r(s),l(s)}=
y_{r(s),l(s)-1}$. If otherwise $y_{r(s),l(s)-1}\in L_{s-1},$ then choose $y_{r(s),l(s)}\notin L_{s-1}$ such that
$$
\|y_{r(s),l(s)}\|\le 1 \qquad \text{and} \qquad \bigl\|y_{r(s),l(s)}-y_{r(s),l(s)-1}\bigr\|< \rho_s.
$$

In both cases $y_{r(s),l(s)}\notin L_{s-1},$ so that we can  set
$$
b_{n_s}=\frac{(I-P_{L_{s-1}})y_{r(s),l(s)}}{\|(I-P_{L_{s-1}})y_{r(s),l(s)}\|}\cdot\sqrt{\eta},
$$
where $P_{L_{s-1}}$ is the orthogonal projection onto $L_{s-1}.$
\medskip

(B) For $n_{s-1}<n\le n_s$ and $i$ such that $(n,i)\in B$, using Lemma \ref{lemma1}, define inductively vectors $v_{ni}, z_{ni}\in H$ such that
\begin{align}\label{list11}
\|v_{ni}\|^2&=\frac{2\sqrt{2}}{C},\qquad\|z_{ni}\|^2\le\frac{2\sqrt{2}}{C},\notag \\
v_{ni}, z_{ni}&\perp^{(T)} v_{m,i'},z_{m,i'},\quad m\le n_s, (m,i')\in B, (m,i')\ne(n,i),\notag \\
v_{ni}, z_{ni}&\perp^{(T)} w_{m},\qquad  m\le n_{s-1}, \\
v_{ni}, z_{ni}&\perp^{(T)} b_{m},\qquad  m\le n_s,\notag \\
v_{ni}, z_{ni}&\perp y_{r,l},\qquad 2^r(2l-1)\le s,\notag\\
z_{ni}&\perp v_{ni},\notag \\
\langle Tv_{ni}&, z_{ni}\rangle=1.\notag
\end{align}

(C) For $n_{s-1}< n\le n_s$ find inductively vectors $w_n\in H$ such that
\begin{align}\label{list12}
\|w_n\|&=1,\notag \\
w_n&\perp^{(T)}v_{m,i},z_{m,i}, \qquad m\le n_s, (m,i)\in B,\notag \\
w_n&\perp^{(T)} w_{m}, \qquad m\le n_s, m\ne n, \\
w_n&\perp^{(T)} b_{m}, \qquad m\le n_s,\notag \\
w_n&\perp  y_{r,l},\qquad 2^r(2l-1)\le s.\notag
\end{align}

(D) If $n$ is such that $n_{s-1}<n\le n_s,$ then 
\begin{align*}
&\sum_{j:(j,n)\in B}|\beta_{jn}|^2\cdot\frac{2\sqrt{2}}{C}+\Bigl\|\sum_{i:(n,i)\in B}\gamma_{ni}z_{ni}+b_n\Bigr\|^2\\
&\le \frac{2\delta\sqrt{2}}{C}+\sum_{i:(n,i)\in B}|\gamma_{ni}|^2\|z_{ni}\|^2+\|b_n\|^2\\
&\le
\frac{4\delta\sqrt{2}}{C}+\eta\le 1.
\end{align*}
Thus we can set
$$
\alpha_n=\Bigl(1-\sum_{j:(j,n)\in B}|\beta_{jn}|^2\cdot\frac{2\sqrt{2}}{C}-\Bigl\|\sum_{i:(n,i)\in B}\gamma_{ni}z_{ni}+b_n\Bigr\|^2\Bigr)^{1/2}.
$$

For $n\le n_s$ set 
$$
u_{n,s}=\alpha_nw_n+\sum_{j:(j,n)\in B\atop j\le n_{s}}\beta_{jn}v_{jn}+\sum_{i:(n,i)\in B}\gamma_{ni}z_{ni}+b_n.
$$
Observe that
\begin{align*}
\|u_{n,s}\|^2&=
\alpha_n^2+\sum_{j\le n_s\atop (j,n)\in B}|\beta_{jn}|^2\cdot\frac{2\sqrt{2}}{C}
+\Bigl\|\sum_{i:(n,i)\in B}\gamma_{ni}z_{ni}+b_n\Bigr\|^2
\\
&=
1-\sum_{j>n_s\atop (j,n)\in B}|\beta_{jn}|^2\cdot\frac{2\sqrt{2}}{C}
=1-\sum_{j>n_s\atop (j,n)\in B}|a_{jn}|\cdot\frac{2\sqrt{2}}{C}
.
\end{align*}
Moreover, in view of \eqref{list11} and \eqref{list12}, the vectors $u_{1, s},\dots, u_{n_s, s}$ are mutually orthogonal.

This finishes our inductive construction.
\medskip

(E) Suppose now that the vectors $b_n,w_n, v_{ni},z_{ni}, n\in\NN, (n,i)\in B,$ and $y_{r,l}, r,l\ge 0$ have been constructed in the way described above. For $n\in\NN$ set
$$
u_n=\alpha_nw_n+\sum_{j:(j,n)\in B}\beta_{jn}v_{jn}+\sum_{i:(n,i)\in B}\gamma_{ni}z_{ni}+b_n,
$$
and note that 
$$u_n=\lim_{s\to\infty} u_{n,s}.$$
 So
$$
\|u_n\|^2=\lim_{s\to\infty}\|u_{n,s}\|^2=
\lim_{s\to\infty} \Bigl(1-\sum_{j>n_s\atop (j,n)\in B}|a_{jn}|\frac{2\sqrt{2}}{C}\Bigr)=1
$$
for every $n \in \mathbb N.$ Moreover, for any $m \in \mathbb N, m\ne n,$ we have
$$
\langle u_m,u_n\rangle=
\lim_{s\to\infty}\langle u_{m,s},u_{n,s}\rangle=0.
$$
Hence  $(u_n)_{n=1}^\infty$ is an orthonormal system in $H$.

\medskip

(F) Next we show that $\langle Tu_m,u_n\rangle=a_{nm},  (n,m) \in B,$
so that  $T$ will have the required matrix with respect to $(u_n)_{n=1}^{\infty}$
after we prove that  $(u_n)_{n=1}^\infty$ is a basis.

Fix $(n,m)\in B,$ and note that  $m<n$. To evaluate  $\langle Tu_m,u_n\rangle,$ decompose it as follows: 
$$
\langle Tu_m,u_n\rangle= A_1+A_2+A_3+A_4+A_5+A_6+A_7,
$$
where
\begin{align*}
A_1=&\Bigl\langle \alpha_mTw_m+\sum_{j': (j',m)\in B}\beta_{j',m}Tv_{j',m}, \alpha_n w_n+\sum_{j:(j,n)\in B}
\beta_{jn}v_{jn}\Bigr\rangle,\\
A_2=&\Bigl\langle \alpha_mTw_m+\sum_{j': (j',m)\in B}\beta_{j',m}Tv_{j',m},\sum_{i: (n,i)\in B}\gamma_{ni}z_{ni}\Bigr\rangle,\\
A_3=&
\Bigl\langle \sum_{i: (m,i)\in B}{\gamma}_{m,i}Tz_{m,i},\alpha_n w_n+\sum_{j:(j,n)\in B}
\beta_{jn}v_{jn}\Bigr\rangle,\\
A_4=&\Bigl\langle \sum_{i': (m,i')\in B}\gamma_{m,i'}Tz_{m,i'},\sum_{i: (n,i)\in B}\gamma_{ni}z_{ni}\Bigr\rangle,\\
A_5=&\Bigl\langle Tb_m,\alpha_n w_n+\sum_{j:(j,n)\in B}
\beta_{jn}v_{jn}\Bigr\rangle,\\
A_6=&\Bigl\langle Tb_m,\sum_{i:(n,i)\in B}\gamma_{ni}z_{ni}\Bigr\rangle,\\
A_7=&\langle Tu_m,b_n\rangle.
\end{align*}
Using the properties given in \eqref{list11} and \eqref{list12}, it is direct to verify that 
$$A_1=A_3=A_4=A_5=A_6=0.$$
To evaluate $A_7$, note that clearly $A_7=0$ if $n\notin \{n_1,n_2,\dots\}$. If otherwise $n=n_s$, then $(n,m)\in B$ implies that $m>n_{s-1}$. In particular, $b_m=0$, so 
$$
Tu_m\in\bigvee
\{Tw_m, Tv_{jm},Tz_{mi}: (j,m)\in B, (m,i)\in B\}\subset \{b_n\}^{\perp}
$$
by our construction.
Hence $A_7=0$ for all $n \in \mathbb N.$ 
Finally, using  \eqref{list11}, we infer that
$$
A_2=\bigl\langle \beta_{nm}Tv_{nm},\gamma_{nm}z_{nm}\bigr\rangle=a_{nm}.
$$
and $\langle Tu_m,u_n\rangle=a_{nm}$ as required.
\medskip

(G) It remains to prove that $(u_n)_{n=1}^\infty$ is a basis in $H$, i.e., that  $(u_n)_{n=1}^\infty$ is complete.
For $n\in\NN$ write $M_n=\bigvee\{u_1,\dots,u_n\}$,
and for every $r\ge 0$  let 
$$y'_r=\lim_{l\to\infty}y_{r,l},$$
where $y_{r, l}$ are defined in Step A and the limit exists by construction and the initial setting (iii).

We show by induction on $l$ that
\begin{equation}\label{11}
\dist^2\bigl\{y_{r,l}, M_{n_{2^r(2l-1)}}\bigr\}\le (1-\eta/2)^{l-1}
\end{equation}
for all $l\in\NN$. This is clear if $l=1$. For $l\ge 2$ let $n=n_{2^r(2l-1)}$ and suppose \eqref{11} is true with $l$ replaced by $l-1$. 
Then, using the definition of $u_n$ along with \eqref{list11} and \eqref{list12}, 
we have
\begin{align*}
\dist^2\{y_{r,l}, M_n\}=&
\dist^2\{y_{r,l},M_{n-1}\}-|\langle y_{r,l},u_n\rangle|^2\\
=&
\dist^2\{y_{r,l},M_{n-1}\}-|\langle y_{r,l},b_n\rangle|^2,
\end{align*}
by the construction of $u_n$.
Moreover, 
$$
|\langle y_{r,l},b_n\rangle|^2=\eta\|(I-P_{L_{s-1}})y_{r,l}\|^2=\eta\|(I-P_{M_{n-1}})y_{r,l}\|^2,
$$
where $s=2^r(2l-1)$. Thus, by \eqref{dist_bound},
\begin{align*}
\dist^2\{y_{r,l}, M_n\}=&
\|(I-P_{M_{n-1}})y_{r,l}\|^2(1-\eta)\\
\le&
\bigl(\|(I-P_{M_{n-1}}y_{r,l-1}\|+\|y_{r,l}-y_{r,l-1}\|\bigr)^2 (1-\eta)\\
\le&
\bigl( (1-\eta/2)^{(l-2)/2}+\rho_s\bigr)^2(1-\eta)\le (1-\eta/2)^{l-1}.
\end{align*}

So, for every $r \ge 0,$
$$
\dist^2\Bigl\{y'_r,\bigvee_{n=1}^\infty u_n\Bigr\}=
\lim_{l\to\infty}\dist^2\bigl\{y_{r,l},M_{2^r(2l-1)}\bigr\}\le
\lim_{l\to\infty}(1-\eta/2)^{l-1}=0.
$$
Hence $y'_r\in\bigvee_{n=1}^\infty u_n$ for all $r \ge 0.$
Now using \eqref{ros} observe that
$$
\sum_{r=0}^\infty \|y_r'-y_r\|
\le\sum_{r=0}^\infty\sum_{l=0}^\infty \|y_{r,l+1}-y_{rl}\|<
\sum_{s=1}^\infty \rho_s<1.
$$
Then by a standard perturbation result for bases, see e.g. \cite[Theorem 1.3.9]{Albiac},
$(y'_m)_{m=1}^\infty$ is a (not necessarily orthonormal) basis in $H.$
Therefore,
$\bigvee_{m=0}^\infty y'_m=H,$ and then $\bigvee_{n=1}^\infty u_n=H$ as well. Thus, the vectors $(u_n)_{n=1}^\infty$ form an orthonormal basis of $H$.
\end{proof}

As noted in the introduction, Theorem \ref{theorem2_intro} can be formulated also for $k$-tuples $\T=(T_1,\dots,T_k)\in B(H)^k$, where none of the operators $T_1,\dots,T_k$ is of the form $\lambda I+K$ with $\lambda\in\CC$ and $K\in B(H)$ compact. The arguments given above can be easily adapted to the multi-operator setting. So the proof remains unchanged and we omit it.

\section{Prescribing matrix entries: general case}

In this section we address a more general setting of operator tuples $\T=(T_1,\dots,T_k) \in B(H)^k$ 
and prescribing elements within the whole of matrix representation for $\T$ rather than its sub (or upper-)diagonal.
We start with proving counterparts of Lemma \ref{lemma1}.  
Necessarily they are a bit more involved though based on the same idea.

\begin{lemma}\label{lemma3}
Let $\T=(T_1,\dots,T_k)\in B(H)^k$, let $0 \in\Int W_e(T),$ and assume that $$\dist\{0,\partial W_e(\T)\}>\e>0,$$
for some $\e >0.$ Then for any subspace  $M\subset H$ of finite codimension,  
there exist vectors $v,z\in M$ satisfying the following conditions:
\begin{itemize}
\item[(a)] $\|v\|=1$;

\item[(b)] $\|z\|\le \frac{2}{\e}$;

\item[(b)] $\langle \T v,v\rangle= 0$;

\item[(c)] $z\in\bigvee_{j=1}^k\{T_jv,T_j^*v\}$;

\item[(d)] $z\perp v$;

\item[(e)] $\langle T_1v,z\rangle=1$, 
$\langle T_1^*v,z\rangle=0$, and
$$\langle T_jv,z\rangle=\langle T_j^*v,z\rangle=0 \qquad \text{for all}\,\, j=2,\dots,k.$$
\end{itemize}
\end{lemma}

\begin{proof}
Denote for short $\T^{-1}M:=\{x \in H: T_j x \in M, 1 \le j \le k\}.$ 
Since $(\pm \e,0,\dots,0)$ and $(\pm i\e,0,\dots,0)\in\Int W_e(\T),$  by Lemma \ref{codim}, (b)
there exists a unit vector $x_1\in M\cap \T^{-1}M\cap \T^{*-1}M$
such that 
$$
\langle \T x_1,x_1\rangle=(\e,0,\dots,0).
$$
Similarly, there exists a unit vector 
$$x_2\in M\cap \T^{-1}M\cap (\T^{*})^{-1}M\cap\{u\in H: u\perp^{(\T)} x_1\}$$
 such that 
$$
\langle \T x_2,x_2\rangle=(i\e,0,\dots,0),
$$ 
and there exist unit vectors $$x_3\in M\cap \T^{-1}M\cap (\T^{*})^{-1}M\cap\{u\in H: u\perp^{(\T)} x_1,x_2\}$$ and $$x_4\in M\cap \T^{-1}M\cap (\T^{*})^{-1}M\cap\{u\in H: u\perp^{(\T)} x_1,x_2,x_3\}$$ with 
$$
\langle \T x_3,x_3\rangle=(-\e,0,\dots)
$$ 
and 
$$
\langle \T x_4,x_4\rangle=(-i\e,0,\dots,0).
$$

Let $$v=\frac{1}{2}(x_1+x_2+x_3+x_4).$$ Then $v\in M\cap \T^{-1}M\cap (\T^{*})^{-1}M$, $\|v\|=1$ and 
$$
\langle \T v,v\rangle=
\frac{1}{4}\Bigl(\langle \T x_1,x_1\rangle+\langle \T x_2,x_2\rangle+\langle \T x_3,x_3\rangle+\langle \T x_4,x_4\rangle\bigr)=0.
$$

Define
$$
L:=\bigvee_{j=1}^k\{T_jv,T_j^*v\}
$$
and
$$
L':=\bigvee\{T_2v,\dots,T_kv,T_1^*v,\dots,T_k^*v\}.
$$
If we let $$
u:=\alpha T_1^*v+\sum_{j=2}^k(\beta_j T_jv+\gamma_j T^*_jv),
$$
then $u\in L',$ and we have
$$
\|T_1v+u\|\ge
|\langle T_1v+u,x_1\rangle|=
\Bigl|\frac{\langle T_1x_1,x_1\rangle}{2}+\alpha\langle T_1^*x_1,x_1\rangle\Bigr|=
\Bigl|\frac{\e}{2}+\alpha\e\Bigr|.
$$
Similarly,
$$
\|T_1v+u\|\ge
|\langle T_1v+u,x_2\rangle|=
\Bigl|\frac{\langle T_1x_2,x_2\rangle}{2}+\alpha\langle T_1^*x_2,x_2\rangle\Bigr|=
\Bigl|\frac{i\e}{2}-i\e\alpha\Bigr|.
$$
So 
$$
\|T_1v+u\|\ge \e\max\Bigl\{\bigl|\frac{1}{2}+\alpha\bigr|,\bigl|\frac{1}{2}-\alpha\bigr|\Bigr\}\ge\frac{\e}{2}
$$
and $\dist\{T_1v,L'\}\ge\frac{\e}{2}$.

Denoting by $P_{L'}$ the orthogonal projection from $L$ onto $L'$,
set finally
$$
z=\frac{(I-P_{L'})T_1v}{\|(I-P_{L'})T_1v\|^2}.
$$
We have $\|(I-P_{L'})T_1v\|=\dist\{T_1v,L'\}\ge\frac{\e}{2}$ and
$\|z\|\le\frac{2}{\e}$.

Moreover, $\langle T_1 v,z\rangle=1$ and
$z\perp L'$. Thus, $v$ and $z$ satisfy all of the conditions (a)--(e). $\Box$
\end{proof}

Next we generalise  Lemma \ref{lemma3} by taking into account 
all of the elements $T_1, \dots , T_k$ of the tuple $\T$.

\begin{lemma}\label{lemma4}
Let $\T=(T_1,\dots,T_k)\in B(H)^k$, and let $\lambda\in\Int W_e(\T)$. 
For any  $0<\e<\dist\{\lambda,\partial W_e(\T)\}$ and
subspace $M\subset H$ of finite codimension,  
there exist vectors $v_1,\dots,v_k, \tilde v_1,\dots, \tilde v_k, z_1,\dots,z_k, \tilde z_1,\dots, \tilde z_k\in M$ satisfying the following conditions for all $j=1,\dots,k:$
\begin{itemize}
\item[(a)] $\|v_j\|\cdot\|z_j\|\le\frac{2}{\e}$ and $\|\tilde v_j\|\cdot\|\tilde z_j\|\le \frac{2}{\e}$;

\item[(b)] $\langle \T v_j,v_j\rangle=\lambda \|v_j\|^2$ and  $\langle \T \tilde v_j, \tilde v_j\rangle=\lambda \|\tilde v_j\|^2$;

\item[(c)]
if $i\ne j, 1 \le i,j \le k,$ then 
\begin{align*}
v_j, \tilde v_j, z_j, \tilde z_j&\perp^{(\T)} v_i, \tilde v_i, z_i, \tilde z_i,\\
\tilde v_j, \tilde z_j&\perp^{(\T)} v_j, z_j,\\
z_r\in&\bigvee_{j=1}^k\{T_j v_r,T_j^*v_r\},\\
\tilde z_r\in&\bigvee_{j=1}^k\{T_j \tilde v_r,T_j^* \tilde v_r\},\\
z_j&\perp v_j, T^*_jv_j,T_rv_j, T^*_r v_j, \quad r\ne j,\\
\tilde z_j&\perp \tilde v_j, T_j\tilde v_j,T_r \tilde v_j, T^*_r \tilde v_j\quad r\ne j,
\end{align*}
\item[(d)] $\langle T_j v_j,z_j\rangle=\langle T_j^*\tilde v_j, \tilde z_j\rangle=1.$
\end{itemize}
\end{lemma}

\begin{proof}

Without loss of generality we may assume that $\lambda=0$. Otherwise, replace $\T$ by $\T-\lambda I$ and construct the vectors $\{v_i, \tilde v_i, z_i, \tilde z _i: 1 \le i \le k\}$ for the $k$-tuple $\T-\lambda I$. Since the vectors
$v_1,\dots,v_k, \tilde v_1,\dots, \tilde v_k, z_1,\dots,z_k, \tilde z_1,\dots, \tilde z_k$ are mutually orthogonal, these vectors satisfy all the conditions required for the $k$-tuple $\T$ as well.

By Lemma \ref{lemma3}, construct vectors $v_1,z_1\in M$ satisfying conditions (a)--(d).

Consider the $k$-tuple $(T_1^*,T_2,\dots,T_k)$. By Lemma \ref{lemma3}, find vectors $$\tilde v_1, \tilde z_1\in M\cap\{u\in H: u\perp^{(\T)} v_1,z_1\}$$ satisfying the conditions (a)--(d).

Consider now the $k$-tuple $(T_2, T_3,\dots,T_k,T_1)$ and  using Lemma \ref{lemma3} again construct vectors $$v_2, z_2\in M\cap\{u\in H: u\perp^{(\T)} v_1,z_1,\tilde v_1, \tilde z_1\}$$ satisfying the conditions (a)--(d).

Continuing this procedure for tuples $$(T^*_2,T_3,\dots,T_k,T_1), (T_3,\dots,T_k,T_1,T_2), \dots, (T_k^*,T_1,\dots,T_{k-1})$$ we construct a family $\{v_i, \tilde v_i, z_i, \tilde z _i: 1 \le i \le k\}\subset H$
with the required properties.
\end{proof}

Now we are ready to prove Theorem \ref{theorem5_intro}, which is one of the main results of this paper.
Recall that for  $\lambda=(\lambda_1,\dots,\lambda_k)\in\CC^k$, we denote  $\|\lambda\|_\infty=\max\{|\lambda_1|,\dots,|\lambda_k|\}$.

\begin{proof}[Proof of Theorem \ref{theorem5_intro}.]

Let $\T=(T_1,\dots,T_k)\in B(H)^k.$
Given $\e >0$ and an admissible set $B \subset \mathbb N \times \mathbb N$, let $\delta$ be such that  
$$
0<\delta < \frac{\e\sqrt{\e}}{18k}.$$
For a given array $\{a_{nj}: (n,j)\in B\cup\Delta \}\subset \mathbb C^k$ satisfying the size conditions
$$a_{nn}\in\Int W_e(\T), \quad  \dist\{a_{nn},\partial W_e(\T)\}>\e \quad \text{for all} \,\, n\in\NN,$$
and
$$\sum_{j:(n,j)\in B}\|a_{nj}\|_\infty\le\delta \quad \text{for all}\,\, n\in\NN \quad \text{and}\quad 
\sum_{n: (n,j)\in B}\|a_{jn}\|_\infty\le\delta \quad \text{for all}\,\, j\in\NN,
$$
we construct an orthonormal basis   $(u_n)_{n=1}^\infty$ in $H$ such that 
\begin{equation}\label{matrix22}
\langle \T u_j,u_n\rangle=a_{nj} \qquad
\text{for all}\,\, n,j\in\NN, (n,j)\in B\cup\Delta.
\end{equation}

Similarly to the proof of Theorem \ref{theorem2_intro}, we will use inductive arguments, and for clarification purposes, 
we will divide them into several steps. 

For $(j,n)\in B\cup\Delta$ let $a_{jn}=(a_{jn}^{(1)},\dots,a_{jn}^{(k)})$.

Fix the following initial settings:
\begin{itemize}
\item [(i)] Without loss of generality we can assume that $\|T_t\|\le 1$ for all $t=1,\dots,k$.
\item [(ii)] We can also assume that $B$ is symmetric, i.e., $(j,n)\in B\Longleftrightarrow (n,j)\in B$. Let $B_\Delta:=\{(j,n)\in B, j>n\}$.
\item [(iii)] Fix $\eta$ such that  
$0<\eta \le\delta.$
\item[(iv)]
For $s\in\NN$ write $s=2^{r(s)}(2l(s)-1),$ where the integers $r(s)\ge 0$ and $l(s)\ge 1$
are defined uniquely.
\item [(v)]
Fix positive numbers $\rho_s, s\in\NN,$ such that $\sum_{s=1}^\infty \rho_s<1$ and  \eqref{dist_bound} holds.
\item [(vi)]
Set formally $n_0=0$. Fix an increasing sequence $(n_s)_{s=1}^\infty$ such that $n_1=1$ and
$(n_s,1),(n_s,2),\dots,(n_s,n_{s-1})\notin B_\Delta$ for all $s \ge 2.$
\item [(vii)]
For $(n,i)\in B_\Delta$ and $t\in\{1,\dots,k\}$ let
\begin{align*}
\beta_{ni}^{(t)}=&|a_{ni}^{(t)}|^{1/2}\arg(a_{ni}^{(t)}),\\
\gamma_{ni}^{(t)}=&|a_{ni}^{(t)}|^{1/2},\\
{\tilde \beta}_{ni}^{(t)}=&|a_{in}^{(t)}|^{1/2}\overline{\arg(a_{in}^{(t)})},\\
{\tilde \gamma}_{ni}^{(t)}=&|a_{in}^{(t)}|^{1/2}.
\end{align*}
\end{itemize}

The construction is similar to the one in the subdiagonal case, though technical details deviate at many steps. 
We look for vectors $u_n, n \in \mathbb N,$ forming the required orthonormal basis by writing them in the form
\begin{align*}
u_n=\alpha_nw_n+&\sum_{t=1}^k\sum_{j:(j,n)\in B_\Delta}\bigl(\beta_{jn}^{(t)}v_{jn}^{(t)}+
{\tilde \beta}_{jn}^{(t)}{\tilde v}_{jn}^{(t)}\bigr)\\
+&\sum_{t=1}^k\sum_{i:(n,i)\in B_\Delta}\bigl(\gamma_{ni}^{(t)}z_{ni}^{(t)}+{\tilde \gamma}_{ni}^{(t)}{\tilde z}_{ni}^{(t)})+b_n,
\end{align*}
where $w_n,b_n, v_{jn}^{(t)}, {\tilde v}_{jn}^{(t)}, z_{ni}^{(t)}, {\tilde z}_{ni}^{(t)}\in H$ with $(j,n), (n,i)\in B_\Delta, t=1,\dots,k$.
Similarly to the subdiagonal case, the pairs $v_{j,n}, z_{j,n}$ for $(j,n)\in B_\Delta$ will be constructed in order to have  $\langle T_tu_n,u_j\rangle=a_{jn}^{(t)}$ for all $1 \le t \le k$, and the pairs 
${\tilde v}_{jn}^{(t)}, {\tilde z}_{jn}^{(t)}$ will ensure that $\langle T_t u_j,u_n\rangle=a_{nj}^{(t)}, 1 \le t \le k$.

As in the subdiagonal case, we fix in advance an orthonormal basis $(y_r)_{r=0}^\infty$ in $H$. Then for each $r \ge 0,$ setting $y_{r,0}=y_r,$
we construct a sequence  $\{y_{r,l}: l \ge 0\}$ such that $\lim_{l\to\infty}y_{r,l}=y'_r, $ $y'_r\in \bigvee_{n=1}^\infty u_n$  and $y'_r$ is close to $y_r.$
From here we derive that $\bigvee_{r=0}^{\infty} y'_r=H$, hence $ \bigvee_{n=1}^{\infty} u_n=H$ as well.

Arguing inductively we construct vectors 
\[ b_n, w_n\in H,  \,\, n\in\NN, \quad\text{and} \quad  v_{ni}^{(t)},{\tilde v}_{ni}^{(t)},z_{ni}^{(t)},{\tilde z}_{ni}^{(t)}\in H, \,\, (n,i)\in B_\Delta, t\in\{1,\dots,k\},
\] 
numbers $\alpha_n\ge 0$, and $k$-tuples $\lambda_n \in\Int W_e(\T)$
as follows:

Let $s\in\NN$ and 
suppose that the vectors 
\[ w_n, b_n, v_{ni}^{(t)}, {\tilde v}_{ni}^{(t)}, z_{ni}^{(t)}, {\tilde z}_{ni}^{(t)}\in H, \qquad n\le n_{s-1}, (n,i)\in B_\Delta, t\in\{1,\dots,k\},  
\]
vectors $y_{r,l}, \quad 2^r(2l-1)\le s-1$, 
numbers $\alpha_n\ge 0$ and $k$-tuples $\lambda_n \in \Int W_e(\T),$  $n\le n_{s-1},$ have been already constructed in such a way
that if for $1\le n \le n_{s-1}$ one sets
\begin{align*}
u_{n,s-1}:=&\alpha_nw_m+\sum_{t=1}^k\sum_{j\le n_{s-1}\atop (j,n)\in B_\Delta}\bigl(\beta_{jn}^{(t)}v_{jn}^{(t)}+
{\tilde \beta}_{jn}^{(t)}{\tilde v}_{jn}^{(t)}\bigr)\\
+&\sum_{t=1}^k\sum_{i:(n,i)\in B_\Delta}\bigl(\gamma_{ni}^{(t)}z_{ni}^{(t)}+{\tilde \gamma}_{ni}^{(t)}{\tilde z}_{ni}^{(t)}\bigr)+b_n,
\end{align*}
then the elements 
\begin{equation}\label{orthog}
u_{1,s-1},\dots,u_{n_{s-1},s-1}\quad \text{are mutually orthogonal} 
\end{equation}
and
\begin{equation}\label{norms}
\|u_{n,s-1}\|^2=1-\sum_{t=1}^k\sum_{j>n_{s-1}\atop (j,n)\in B_\Delta}(|a_{jn}|+|a_{nj}|)\cdot\frac{4}{\e^{3/2}}.
\end{equation}
Moreover, we assume that $\|y_{r,l}-y_{r,l-1}\|\le \eta_{2^r(2l-1)}$ for all $r\ge 0,l\ge 1$ with $2^r(2l-1)\le s-1$.
\medskip

(A) We first define $b_n, n_{s-1}<n\le n_s:$

For $n_{s-1}<n<n_s$ set $b_n=0$.
Let $s=2^{r(s)}(2l(s)-1),$ and define  
\[
L_{s-1}:=\bigvee_{n=1}^{n_{s-1}}u_{n,s-1}. 
\]
Arguing as in the proof of Theorem \ref{theorem2_intro}, find $y_{r(s),l(s)}\in H\setminus L_{s-1}$ such that $\|y_{r(s),l(s)}\|\le 1$ and 
\begin{equation}\label{ros}
\bigl\|y_{r(s),l(s)}-y_{r(s),l(s)-1}\bigr\|<\rho_s.
\end{equation}
Then set 
\begin{equation}
b_{n_s}:=\frac{(I-P_{L_{s-1}})y_{r(s),l(s)}}
{\bigl\|(I-P_{L_{s-1}})y_{r(s),l(s)}\bigr\|}\cdot\sqrt{\eta}.
\end{equation}

(B) Next, arguing inductively for $n=n_{s-1}+1,\dots,n_s,$ we construct vectors $v_{ni}^{(t)},{\tilde v}_{ni}^{(t)},z_{ni}^{(t)}, {\tilde z}_{ni}^{(t)}\in H$ and $k$-tuples $\lambda_n\in\Int W_e(\T)$.

Let $n_{s-1}<n<n_s$ and suppose that the vectors $v_{mi}^{(t)},{\tilde v}_{mi}^{(t)},z_{mi}^{(t)}, {\tilde z}_{mi}^{(t)}\in H$ and $\lambda_m\in\Int W_e(\T)$ have been constructed for all $m<n, (m,i)\in B_\Delta$ and $t\in\{1,\dots,k\}$.

Using Lemma \ref{lemma4} repeatedly, find for all $i, (n,i)\in B_\Delta$ and $t=1,\dots,k$ vectors $v_{ni}^{(t)}, {\tilde v}_{ni}^{(t)}, z_{ni}^{(t)}, {\tilde z}_{ni}^{(t)}\in H$ 
such that for all\ \ $m\le n, (n,i),(m,j)\in B_\Delta, 1\le t,t'\le k, (n,i,t)\ne(m, j, t'),$
\begin{align}\label{list1}
\|v_{ni}^{(t)}\|=\|{\tilde v}_{ni}^{(t)}\|&=\frac{2}{\e^{3/4}}, \notag\\
\|z_{ni}^{(t)}\|, \|{\tilde z}_{ni}^{(t)}\|&\le \e^{-1/4},\notag \\
v_{ni}^{(t)}, {\tilde v}_{ni}^{(t)}, z_{ni}^{(t)}, {\tilde z}_{ni}^{(t)}&\perp^{(\T)}  v_{mj}^{(t')}, {\tilde v}_{mj}^{(t')}, z_{mj}^{(t')}, {\tilde z}_{mj}^{(t')}\notag \\
{\tilde v}_{ni}^{(t)}, {\tilde z}_{ni}^{(t)}&\perp^{(\T)} v_{ni}^{(t)}, z_{ni}^{(t)},\notag \\
v_{ni}^{(t)},{\tilde v}_{ni}^{(t)}, z_{ni}^{(t)}, {\tilde z}_{ni}^{(t)}&\perp^{(\T)} b_m, \quad m\le n_s,\\
v_{ni}^{(t)}, {\tilde v}_{ni}^{(t)}, z_{ni}^{(t)}, {\tilde z}_{ni}^{(t)}&\perp^{(\T)} w_m, \quad m\le n_{s-1},\notag \\
v_{ni}^{(t)}, {\tilde v}_{ni}^{(t)}, z_{ni}^{(t)}, {\tilde z}_{ni}^{(t)}&\perp^{(\T)} y_{r,l},\quad2^r(2l-1)\le s,\notag \\
z_{ni}^{(t)}\perp v_{ni}^{(t)}, {\tilde z}_{ni}^{(t)}&\perp {\tilde v}_{ni}^{(t)},\notag \\
\langle \T^* v_{ni}^{(t)}, z_{ni}^{(t)}\rangle&=\langle \T {\tilde v}_{ni}^{(t)},{\tilde z}_{ni}^{(t)}\rangle=0,\notag \\
\langle T_t v_{ni}^{(t')}, z_{ni}^{t'}\rangle&= 
\langle \T^*_t {\tilde v}_{ni}^{(t')},{\tilde z}_{ni}^{(t')}\rangle=\delta_{t,t'}, \notag 
\end{align}
where $\delta_{t,t'}$ is the Kronecker symbol and 
\begin{align*}
\langle \T v_{ni}^{(t)},v_{ni}^{(t)}\rangle=
\lambda_i\|v_{ni}^{(t)}\|^2=\frac{4\lambda_i}{\e^{3/2}}, \qquad
\langle \T\tilde v_{ni}^{(t)},\tilde v_{ni}^{(t)}\rangle
 = \lambda_i\|\tilde v_{ni}^{(t)}\|^2
=\frac{4\lambda_i}{\e^{3/2}}
\end{align*}
(note that $i<n$, and so $\lambda_i\in\Int W_e(\T)$ was already constructed).

Write for short
$$
x_n=\sum_{t=1}^k\sum_{i:(n,i)\in B}\bigl(\gamma_{ni}^{(t)}z_{ni}^{(t)}+{\tilde \gamma}_{ni}^{(t)}{\tilde z}_{ni}^{(t)}\bigr)+b_n.
$$
We have
$$
\|x_n\|^2
=\sum_{t=1}^k\sum_{i:(n,i)\in B_\Delta}(|\gamma^{(t)}_{ni}|^2+|\tilde\gamma^{(t)}_{ni}|^2)+\|b_n\|^2
\le 2k\delta+\eta\le (2k+1)\delta
\le\frac{\e}{6}\le \frac{1}{6}.
$$
Thus if we set
$$
\lambda_n=
\frac{a_{nn}-\langle \T x_n,x_n\rangle}
{1-\|x_n\|^2},
$$
then
$$
\|\lambda_n-a_{nn}\|_\infty\le
\Bigl\|a_{nn}-\frac{a_{nn}}{1-\|x_n\|^2}\Bigr\|_\infty+\Bigl|\frac{\|x_n\|^2}{1-\|x_n\|^2}\Bigr|\le
\frac{2\|x_n\|^2}{1-\|x_n\|^2}\le
\frac{\e/3}{1-\frac{1}{6}}\le \frac{\e}{2}.
$$
So $\lambda_n\in\Int W_e(T)$ and $\dist\{\lambda_n,\partial W_e(T)\}>\e/2$.
\medskip

(C) Suppose that the vectors 
\[ v_{ni}^{(t)},{\tilde v}_{ni}^{(t)}, z_{ni}^{(t)}, {\tilde z}_{ni}^{(t)}, \qquad n\le n_s, (n,i)\in B_\Delta, t=1,\dots,k,\]
 and the $k$-tuples $\lambda_n\in \Int W_e(\T)$ have been constructed.

Choose inductively vectors $w_n, n_{s-1}<n\le n_s,$ satisfying 
\begin{align}\label{list2}
\|w_n\|&=1,\notag \\
\langle\T w_n &,w_n\rangle=\lambda_n,\notag \\
w_n&\perp^{(\T)} v_{mi}^{(t)},{\tilde v}_{mi}^{(t)}, z_{mi}^{(t)}, {\tilde z}_{mi}^{(t)}, \quad
 m\le n_s, (m,i)\in B_\Delta, t=1,\dots,k,\\
w_n&\perp^{(\T)} b_m, \quad m\le n_s, \notag \\
w_n&\perp^{(\T)} w_m, \quad m\ne n,\notag\\
w_n&\perp y_{r,l}\quad 2^r(2l-1)\le s.\notag
\end{align}
\medskip

(D) For every $n$ such that  $n_{s-1}+1 \le n \le n_s$  we have
$$
\sum_{t=1}^k\sum_{j:(j,n)\in B}(|\beta_{jn}^{(t)}|^2+|{\tilde \beta}_{jn}^{(t)}|^2)\frac{4}{\e^{3/2}}+\|x_n\|^2\le
\frac{8k\delta}{\e^{3/2}}+\e/6\le 1.
$$
Define now 
$$
\alpha_n=\Bigl(1-\sum_{t=1}^k\sum_{j:(j,n)\in B}\bigl(|\beta_{jn}^{(t)}|^2+|{\tilde \beta}_{jn}^{(t)}|^2\bigr)\frac{4}{\e^{3/2}}+\|x_n\|^2\Bigr)^{1/2}.
$$

For $n\le n_s$ set
\begin{align*}
u_{n,s}=
\alpha_n w_n+&
\sum_{t=1}^k\sum_{j\le n_s\atop(j,n)\in B_\Delta}\bigl(\beta_{jn}^{(t)}v_{jn}^{(t)}+{\tilde \beta}_{jn}^{(t)}{\tilde v}_{jn}^{(t)}\bigr)\\
+&\sum_{t=1}^k\sum_{i:(n,i)\in B_\Delta}\bigl(\gamma_{ni}^{(t)}z_{ni}^{(t)}+{\tilde \gamma}_{ni}^{(t)}{\tilde z}_{ni}^{(t)}\bigr)+b_n.
\end{align*}
Then the vectors $u_{1,s},\dots, u_{n_s,s}$ are mutually orthogonal and
$$
\|u_{n,s}\|^2=
1-\sum_{t=1}^k\sum_{j>n_s\atop (j,n)\in B_\Delta}(|a_{jn}|+|a_{nj}|)\frac{4}{\e^{3/2}}.
$$
In other words, \eqref{orthog} and \eqref{norms} hold with $s-1$ replaced by $s.$
\medskip

(E) Suppose that the vectors $w_n, b_n,  n\in\NN,$ and $v_{ni}^{(t)},{\tilde v}_{ni}^{(t)}, z_{ni}^{(t)}, {\tilde z}_{ni}^{(t)},  (n,i)\in B_\Delta, t=1,\dots,k,$ 
are constructed. 

Set
\begin{align}\label{defun}
u_n:=
\alpha_n w_n+&
\sum_{t=1}^k\sum_{j:(j,n)\in B_\Delta}\bigl(\beta_{jn}^{(t)}v_{jn}^{(t)}+{\tilde \beta}_{jn}^{(t)}{\tilde v}_{jn}^{(t)}\bigr)\\
+&\sum_{t=1}^k\sum_{i:(n,i)\in B_\Delta}\bigl(\gamma_{ni}^{(t)}z_{ni}^{(t)}+{\tilde \gamma}_{ni}^{(t)}{\tilde z}_{ni}^{(t)}\bigr)+b_n.\notag
\end{align}
Since
$$
u_n=\lim_{s\to\infty}u_{n,s}
$$
for all $n\in\NN$, we have 
$$
\|u_n\|=\lim_{s\to\infty}\|u_{n,s}\|=1 \qquad \text{for all}\quad n \in\mathbb N.
$$
Moreover, for  all $m,n \in \mathbb N,$ $m\ne n,$ 
$$
\langle u_m,u_n\rangle=\lim_{s\to\infty}\langle u_{m,s},u_{n,s}\rangle=0.
$$
Hence the vectors $(u_n)_{n=1}^\infty$ form an orthonormal system in $H$.

\medskip

(F) Let $(n,m)\in  B_\Delta$ and $t\in\{1,\dots,k\}$. To evaluate the inner product $\langle T_t u_m, u_n\rangle$ we use  definition (\ref{defun}) 
of $(u_n)_{n=1}^{\infty}$ and  decompose $\langle T_t u_{m}, u_n\rangle$
as
$$
\langle T_t u_{m},u_n\rangle= A_1+A_2+A_3+A_4+A_5+A_6+A_7,
$$
where
\begin{align*}
A_1=&
\Bigl\langle \alpha_mT_t w_m+\sum_{t'=1}^k\sum_{j: (j,m)\in B_\Delta} \bigl(\beta_{jm}^{(t')}T_t v_{jm}^{(t')}+{\tilde \beta}_{jm}T_t {\tilde v}_{jm}^{(t')}\bigr),\\
&\alpha_nw_n+\sum_{t''=1}^k\sum_{j':(j'n)\in B_\Delta}\bigl(\beta_{j'n}^{(t'')}v_{j'n}^{(t'')}+{\tilde \beta}_{j'n}^{(t'')}{\tilde z}_{j'n}^{(t'')}\bigr)\Bigr\rangle,\\
A_2=&\Bigl\langle \alpha_mT_t w_m+\sum_{t'=1}^k\sum_{j: (j,m)\in B_\Delta} \bigl(\beta_{jm}^{(t')}T_t v_{jm}^{(t')}+{\tilde \beta}_{jm}T_t {\tilde v}_{jm}^{(t')}\bigr),\\
&\sum_{t''=1}^k\sum_{i:(n,i)\in B}(\gamma_{ni}^{(t'')}z_{ni}^{(t'')}+{\tilde \gamma}_{ni}^{(t'')}{\tilde z}_{ni}^{(t'')})\Bigr\rangle,\\
A_3=&
\Bigl\langle \sum_{t'=1}^k\sum_{i: (m,i)\in B_\Delta}\bigl(\gamma_{mi}^{(t')}T_t z_{mi}^{(t')}+{\tilde \gamma}_{mi}^{(t')}T_t {\tilde z}_{mi}^{(t')}\bigr),\\
&\alpha_nw_n+\sum_{t''=1}^k\sum_{j':(j'n)\in B_\Delta}\bigl(\beta_{j'n}^{(t'')}v_{j'n}^{(t'')}+{\tilde \beta}_{j'n}^{(t'')}{\tilde z}_{j'n}^{(t'')}\bigr)\Bigr\rangle,
\end{align*}
\begin{align*}
A_4=&
\Bigl\langle \sum_{t'=1}^k\sum_{i: (m,i)\in B_\Delta}\bigl(\gamma_{mi}^{(t')}T_t z_{mi}^{(t')}+{\tilde \gamma}_{mi}^{(t')}T_t {\tilde z}_{mi}^{(t')}\bigr),\\
&\sum_{t''=1}^k\sum_{i:(n,i)\in B}\bigl(\gamma_{ni}^{(t'')}z_{ni}^{(t'')}+{\tilde \gamma}_{ni}^{(t'')}{\tilde z}_{ni}^{(t'')}\bigr)\Bigr\rangle,\\
A_5=&
\Bigl\langle T_tb_m,
\alpha_nw_n+\sum_{t''=1}^k\sum_{j':(j'n)\in B_\Delta}(\beta_{j'n}^{(t'')}v_{j'n}^{(t'')}+{\tilde \beta}_{j'n}^{(t'')}{\tilde z}_{j'n}^{(t'')})\Bigr\rangle,\\
A_6=&
\Bigl\langle T_tb_m,
\sum_{t''=1}^k\sum_{i:(n,i)\in B}\bigl(\gamma_{ni}^{(t'')}z_{ni}^{(t'')}+{\tilde \gamma}_{ni}^{(t'')}{\tilde z}_{ni}^{(t'')}\bigr)\Bigr\rangle,\\
A_7=&\langle T_tu_m,b_n\rangle.
\end{align*}
By construction,  using the properties of $w_m, v_{ni},$ and $ z_{ni}$  listed in \eqref{list1} and \eqref{list2}, 
we have 
$$A_1=A_3=A_4=A_5=A_6=0.$$ 
Similarly,
$$
A_2=\bigl\langle \beta_{nm}^{(t)}T_tv_{nm}^{(t)},\gamma_{nm}^{(t)}z_{nm}^{(t)}\bigr\rangle=a_{nm}^{(t)}.
$$

It remains to consider $A_7.$ Clearly $A_7=0$ if $n\notin\{n_1,n_2,\dots\}$. Suppose that $n=n_s$ for some $s \in \mathbb N$. Since $(n,m)\in B_\Delta$, we have $m>n_{s-1}$.
So $b_m=0$ and
$$
T_tu_m\in\bigvee\Bigl\{T_tw_m,T_tv_{jm}^{(t')}, T_t{\tilde v}_{jm}^{(t')}, T_tz_{mi}^{(t')}, T_t{\tilde z}_{mi}^{(t')}:
$$
$$
\hskip2cm (m,i)\in B_\Delta,(j,m)\in B_\Delta, t'=1,\dots,k\Bigr\}\subset \{b_n\}^{\perp}.
$$
So $A_7=0.$ Thus, summarising the above,
$$
\bigl\langle T_tu_m,u_n\bigr\rangle=a_{nm}^{(t)}.
$$
Similarly,
$$
\langle T_tu_n,u_m\rangle=
\overline{\langle T^*_tu_m,u_n\rangle}
=\overline{{\tilde \beta}_{nm}^{(t)}T_t {\tilde v}_{nm}^{(t)},{\tilde \gamma}_{nm}^{(t)}{\tilde z}_{nm}^{(t)}\rangle}
=a_{m,n}^{(t)}.
$$
\medskip

(G) Finally, for each $n\in \NN$, we compute the diagonal elements $\langle\T u_n,u_n\rangle$. We have
\begin{align*}
\langle\T u_n,u_n\rangle
=&
\alpha_n^2\langle \T w_n,w_n\rangle+
\sum_{t=1}^k\sum_{j:(j,n)\in B_\Delta}\Bigl(\langle \T v_{jn}^{(t)},v_{jn}^{(t)}\rangle\|v_{jn}^{(t)}\|^2\\
+&\langle \T v_{jn}^{(t)},v_{jn}^{(t)}\rangle\Bigr)\|v_{jn}^{(t)}\|^2+\langle \T x_n,x_n\rangle\\
=&\lambda_n\Bigl(\alpha_n^2+\sum_{t=1}^k\sum_{j:(j,n)\in B_\Delta}(|a_{jn}|+|a_{nj}|)\frac{4}{\e^{3/2}}\Bigr)+\langle \T x_n,x_n\rangle\\
=&\lambda_n(1-\|x_n\|^2)+\langle \T x_n,x_n\rangle\\
=&a_{nn}.
\end{align*}
Thus, $(u_n)_{n=1}^{\infty}$ satisfies \eqref{matrix22}.
\medskip

(H) 
It remains to show that $(u_n)_{n=1}^{\infty}$ is a basis in $H$. 
The argument in this step is analogous to the corresponding argument in the proof of
Theorem \ref{theorem2_intro}, step G, and its sketch below is given for completness.

Fix $r\ge 0$ and let 
$$y'_r=\lim_{l\to\infty}y_{r,l},$$
where the sequence $\{y_{r,l}: l \ge 0 \}$ is defined in Step A, and the limit exists in view of the initial setting (v). 
Setting $M_n=\bigvee \{u_1, \dots, u_n\}, n \in \mathbb N,$ we show by induction on $l$ that
\begin{equation}\label{1}
\dist^2\bigl\{y_{r,l}, M_{n_{2^r(2l-1)}}\bigr\}\le (1-\eta/2)^{l-1}
\end{equation}
for all $l\in\NN$. Since the inequality is obvious if $l=1$, we let $l\ge 2$ and $n=n_{2^r(2l-1)}.$ Assuming that \eqref{1} is true with $l$ replaced by $l-1$, 
as in the step G from the proof of Theorem \ref{theorem2_intro}, we infer that 
\begin{align*}
\dist^2\{y_{r,l}, M_n\}=&
\dist^2\{y_{r,l},M_{n-1}\}(1-\eta)\\
\le&
\bigl( (1-\eta/2)^{(l-2)/2}+\rho_s\bigr)^2(1-\eta)\\
\le& (1-\eta/2)^{l-1}.
\end{align*}
From here it follows that
$$
\dist^2\Bigl\{y'_r,\bigvee_{n=1}^\infty u_n\Bigr\}=
\lim_{l\to\infty}\dist^2\bigl\{y_{r,l},M_{2^r(2l-1)}\bigr\}\le
\lim_{l\to\infty}(1-\eta/2)^{l-1}=0,
$$
hence $y'_r\in\bigvee_{n=1}^\infty u_n$.
Since, taking into account \eqref{ros},
$$
\sum_{r=0}^\infty \|y_r'-y_r\|\le
\sum_{r=0}^\infty\sum_{l=0}^\infty\|y_{r,l+1}-y_r\|<\sum_{s=1}^\infty \rho_s<1,
$$ 
we conclude, similarly to the proof of Theorem \ref{theorem2_intro}, step G,
that $(y'_r)_{r=0}^\infty$ is a basis by \cite[Theorem 1.3.9]{Albiac}.
Hence
$\bigvee_{n=1}^\infty u_n=H$,  so that the vectors $(u_n)_{n=1}^\infty$ form an orthonormal basis in $H.$ This finishes the proof.
\end{proof}

\section{Tuples consisting of powers of a single operator}\label{section_powers}

Theorem \ref{theorem5_intro} can be directly applied to any $k$-tuple $(T,T^2,\dots,T^k)$, where $T\in B(H)$ satisfies $\Int\hat\sigma(T)\ne\emptyset$. By Theorem \ref{spectrum_polyn}, 
$$
(\lambda,\lambda^2,\dots,\lambda^k)\in\Int W_e(T,\dots,T^k)
$$ 
for all $\lambda\in\Int\hat\sigma(T)$. In particular we have the following corollary.

\begin{corollary}\label{hull}
Let $T\in B(H)$ satisfy $0\in\Int\hat\sigma(T)$, and $k\in\NN$. 
Then there exists $\delta>0$ with the following property: if $B\subset(\NN\times\NN)\setminus\Delta$ is admissible, and 
$\{a_{nj}=(a_{nj}^{(1)},\dots,a_{nj}^{(k)}): (n,j)\in B\cup\Delta\}\subset \mathbb C^k$ 
is such that 
$\sum_{j:(n,j)\in B\cup\Delta}\|a_{nj}\|_\infty\le\delta$
for all $n\in\NN$,
and $\sum_{n: (n,j)\in B\cup\Delta}\|a_{jn}\|_\infty\le\delta$ for all $j\in\NN$, 
then there is an orthonormal basis $(u_n)_{n=1}^\infty$ in $H$ satisfying
$$ 
\langle T^l u_j,u_n\rangle=a_{nj}^{(l)}
$$
for all $l=1,\dots,k$ and $n,j\in\NN$ with $(n,j)\in B\cup\Delta$.

In particular, there exists an orthonormal basis $(u_n)_{n=1}^\infty$ in $H$ such that
$$ 
\langle T^l u_j,u_n\rangle=0
$$
for all $l=1,\dots,k$ and $n,j\in\NN$ with $(n,j)\in B\cup\Delta$.
\end{corollary}

It is a standard intuition behind many properties in operator theory that if $T\in B(H)$ is invertible then the  operators $T$ and $T^{-1}$ cannot be "small" at the same time.
However, as far as matrix representations are concerned our technique allows one to get several statements
which seem to oppose this general principle. One statement of this kind
concerns the size of matrix elements  for $T$ and $T^{-1}.$

Recall that as it was proved in \cite[Theorems 2.5 and 6.2]{MT_JFA1}, 
for $\T=(T_1,\dots, T_k)\in B(H)^{k}$ one has  
$0\in W_{e}(\T)$ if and only if
for any $(a_{n})_{n=1}^{\infty }\subset (0,\infty )$ satisfying
$(a_{n})_{n=1}^{\infty }\notin \ell ^{1}(\mathbb{N})$ there exists an orthonormal
basis $(u_{n})_{n=1}^{\infty }$ in $H$ such that
\begin{equation*}
|\langle T_{l} u_{n},u_{j}\rangle |\le \sqrt{a_{n} a_{j}}
\end{equation*}
for all $n,j\in \mathbb N$ and $l=1,\dots,k$.

Then, given an invertible $T \in B(H)$ and  considering the tuple $\T=(T^{-1}, T)$
the next statement was obtained (\cite[Theorem 6.3]{MT_JFA1}).

\begin{theorem} \label{inverse}
Let $T \in B(H)$ be an invertible operator such that there exists
a nonzero $\lambda \in \mathbb{C}$ satisfying
$\{\lambda,-\lambda\} \subset \sigma _{e}(T)$. Then for any
$(a_{n})_{n=1}^{\infty }\subset (0,\infty )$ satisfying
$(a_{n})_{n=1}^{\infty }\notin \ell ^{1}(\mathbb{N})$ there exists an orthonormal
basis $(u_{n})_{n=1}^{\infty } \subset H$ such that
\begin{equation*}
|\langle T u_{n},u_{j}\rangle |\le \sqrt{a_{n} a_{j}} \qquad
\text{and} \qquad |\langle T^{-1} u_{n},u_{j}\rangle |\le \sqrt{a_{n} a_{j}}
\end{equation*}
for all $n,j\in \mathbb N $.
\end{theorem}

Developing this interesting effect a bit further,  we will consider more general $(2k)$-tuples $\T$
of the form
\begin{equation}\label{2kei}
\T=(T^{-k},\dots,T^{-1},T,\dots,T^k), \quad k \in \mathbb N,
\end{equation}
where $T\in B(H)$ is an invertible operator. 
We will be interested in providing sparse
matrix representations for the elements of $\T$ with respect to the same basis and 
for the same admissible set $B \subset \mathbb N 
\times \mathbb N.$ It is instructive to recall that the set $B$ can be quite large.

To deduce them from Theorem \ref{theorem5_intro},
note first that 
\begin{equation}\label{inclus}
\sigma_e(\T)=\{(\lambda^{-k},\dots,\lambda^{-1},\lambda,\dots,\lambda^k): \lambda\in\sigma_e(T)\}
\end{equation}
 by the spectral mapping theorem for the essential spectrum,
see \cite[Corollary 30.11]{Muller}.
Since the  spectral mapping theorem in this generality is quite involved we offer a simple argument for our particular situation.
To this aim and to make a link to the studies in \cite{MT_JFA1}, recall that  
if $(T_1,\dots, T_k)\in B(H)^{k}$ then
$(\lambda_1, \dots, \lambda_k) \in  \sigma _{e}((T_1, \dots, T_k))$ if and only if there exists a sequence
of unit vectors $(x_{n})_{n=1}^{\infty }\in H$ converging weakly to zero such that 
either
$\|T_{l} x_{n} - \lambda _{l} x_{n}\|\to 0$, $1 \le l \le k,$
or $\|T^{*}_{l} x_{n} - \bar \lambda _{l} x_{n}\|\to 0$,
$1 \le l \le k,$ as $n \to \infty $,   see e.g. \cite[p. 122-123]{Dash}.
Rather than $\|x_n\|=1, n \in \mathbb N,$ it suffices to require $\inf_n\|x_n\|>0.$
The next elementary lemma 
shows that the set defined in right hand side of \eqref{inclus} is contained
in $\sigma_e(\T).$

\begin{lemma}\label{inverse1}
Let $T\in B(H)$ be an invertible operator, $k\in\NN$, and let $\T$ be given by \eqref{2kei}.
If $\la\in\sigma_e(T),$ then 
$$(\la^{-k},\dots,\la^{-1},\la,\dots,\la^k)\in \sigma_e(\T).$$
\end{lemma}

\begin{proof}
Let  $(x_n)_{n=1}^{\infty} \subset H$ be a sequence converging weakly to $0$ such that 
$\inf_{n \ge 1}\|x_n\|>0$ and 
either $\|(T-\la)x_n\|\to 0$ or $\|(T^*-\bar\la)x_n\|\to 0$ as $n \to \infty.$

In the first case let $y_n=\la^k T^kx_n,  n \ge 1$. Then $\inf_n\|y_n\|>0$ since $T$ is invertible, $y_n\to 0, n \to \infty,$ weakly and 
$$(T^l-\la^l)y_n=\la^k T^l(T^{l-1}+\la T^{l-2}+\cdots+\la^{l-1})(T-\la)x_n\to 0$$
 for all $l=1,\dots,k$. Furthermore,
$(T^{-l}-\la^{-l})y_n=\la^{k-l}T^{k-l}(\la^l-T^l)x_n\to 0$ for all $l=1,\dots,k$. So
$$
(\la^{-k},\dots,\la^{-1},\la,\dots,\la^k)\in\sigma_e(\T).
$$
In the second case set $z_n=\bar\la^k T^{*k}x_n, n \ge 1$. We have similarly $(T^{*l}-\bar\la^l)z_n\to 0, n \to \infty,$ for all $l=\pm 1,\dots,\pm k$, 
hence again $(\la^{-k},\dots,\la^{-1},\la,\dots,\la^k)\in\sigma_e(\T)$.

\end{proof}

To be able to apply Theorem \ref{theorem5_intro} to  $(2k)$-tuples $\T$ of the form \eqref{2kei},
 we should ensure that $0\in\Int W_e(\T),$ and the statement below provides $\T$ with this property under natural geometric spectral assumptions on $T$.

\begin{proposition}\label{ituple}
Let $T\in B(H)$ be an invertible operator, and let $s>r>0$ be fixed. Suppose that $\{re^{i\varphi}, se^{i\varphi}: 0\le \varphi< 2\pi\}\subset\sigma_e(T)$. 
Let $k\in\NN$ and $\T=(T^{-k},\dots,T^{-1},T,\dots,T^k)\in B(H)^k$. 
Then
$0\in\Int W_e(\T).$
\end{proposition}

\begin{proof}
By Theorem \ref{spectrum_polyn}, 
there exists $a>0$ such that if $z=(z_1,\dots, z_k)\in \mathbb C^k$ satisfies $\|z\|_\infty<a,$ then $z \in \conv\{(\la,\dots,\la^k):|\la|=r\}$. 
So for all $z=(z_1, \dots, z_k)$ with $\|z\|_\infty \le a,$ using Lemma \ref{inverse1}, one has
$$
\Bigl(\frac{{\bar z_k}^k}{r^{2k}},\dots,\frac{\bar z_1}{r^2},z_1,\dots,z_k\Bigr)\in \conv\sigma_e(\T).
$$

Similarly, there exists $a'>0$ such that all $z=(z_1, \dots, z_k) \in \CC^k$ with $\|z\|_\infty \le a'$ satisfy $z \in\conv\{\la,\dots,\la^k):|\la|=s\}$, hence
$$
\Bigl(\frac{-\bar z_k}{s^{2k}},\dots,\frac{-\bar z_1}{s^2},-z_1,\dots,-z_k\Bigr)\in\conv \sigma_e(\T).
$$

Therefore, for all $z=(z_1,\dots,z_k)\in\CC^k$ with $\|z\|_\infty \le\min\{a,a'\}$ we have
$$
\Bigl(\frac{\bar z_k}{2r^{2k}}-\frac{\bar z_k}{2s^{2k}},\dots, \frac{\bar z_1}{2r^2}-\frac{\bar z_1}{2s^2},0,\dots,0\Bigr)\in \conv \sigma_e(\T)\subset W_e(\T).
$$
since $W_e(\T)$ is a convex set. 
So $(z_{-k},\dots, z_{-1},0,\dots, 0)\in W_e(\T)$ for all $z_{-k},\dots,z_{-1}\in \CC$ with sufficiently small $\max\{|z_{-j}|: 1\le j\le k\}.$ 

Similarly, $(0,\dots, 0, z_1,\dots, z_k)\in W_e(\T)$ for all $z_1,\dots,z_k\in\CC$ with $\max\{|z_{j}|: 1\le j\le k\}$ small enough. 
Combining these two inclusions and using the convexity of $W_e(\T)$ again, we obtain that $(0,\dots,0)\in \Int W_e(\T)$.
\end{proof}

This together with Theorem \ref{theorem5_intro} proves the following corollary.

\begin{corollary}\label{inverse_sparse_intro}
Let $T\in B(H)$ be an invertible operator, and $B\subset(\NN\times\NN)\setminus\Delta$ be an admissible set. 
Assume that there exist $s>r>0$ such that $s\mathbb T\cup r\mathbb T \subset \sigma_e(T)$, 
where $\mathbb T$ stands for the unit circle.
Then for every $k \in \mathbb N$ there exists an orthonormal basis $(u_n)_{n\in\NN}$ in $H$ such that
$$
\langle T^lu_j,u_n\rangle=0
$$
for all $l=\pm 1,\dots,\pm k$ and $(n,j)\in B\cup\Delta$.
\end{corollary}

Simple examples enjoying spectral assumptions in  Corollary \ref{inverse_sparse_intro} can be found already within the class of normal operators.
 In particular, one may consider a multiplication operator $(Mf)(z)=zf(z)$ on $L^2(S,d\mu),$ where
$S \subset \mathbb C$ is a Borel set containing $r\mathbb T$ and $s\mathbb T, s >r>0$,
 and $\mu$ is a Borel measure on $S$ whose essential support contains $r\mathbb T$ and $s\mathbb T$ as well.
Of course, more general examples of this kind can be provided by replacing $z$ with a function $g \in L^\infty(S, d \mu).$
Even in this case the existence of sparse representations provided by Corollary \ref{inverse_sparse_intro} is far from being obvious.

Another class of examples of operators fitting  Corollary \ref{inverse_sparse_intro} is provided by invertible composition operators $C$ considered e.g. in \cite{Ridge1}. 
If the spectral radius $r(C)$ of $C$ equals $s$ and  $r(C^{-1})=r^{-1},$ then
by rotation invariance (\cite[Theorem B]{Ridge1}) we have
$\partial \sigma (C)\supset s\mathbb T$ and $\partial \sigma (C^{-1})\supset r^{-1}\mathbb T$ whenever $r$ and $s$ are different from $1.$ 
So  $s\mathbb T \subset \sigma_e(T)$ and $r \mathbb T \subset \sigma_e(T)$ since $s\mathbb T$ and $r\mathbb T$ belong to $\partial \sigma(C).$ 
A concrete example of such a composition operator on $L^2(0,1)$ is considered e.g. 
in \cite[Example (2)]{Ridge1}. A similar and quite general example
now  addressing composition operators  
on the Hardy space $H^2(\mathbb D)$ is analysed
in \cite[Theorem 6]{Nord}.
In the framework of composition operators, the availability of sparse representations seem to be also highly nontrivial 

Remark finally that similar applications of our results can be provided in the setting of subdiagonal arrays $B,$
but we omit easy formulations of the corresponding corollaries in this setting.

\section{Acknowledgments}

We would like to thank the referee for
of pertinent comments and remarks that led to improvement of this paper.

\end{document}